\documentclass[12pt,oneside,british,a4wide]{amsart}
\usepackage{amsmath, amssymb,verbatim,appendix}
\usepackage{hyperref}
\usepackage[mathscr]{eucal}
\usepackage{amscd}
\usepackage{amsthm}
\usepackage{stmaryrd}
\usepackage{enumerate}
\usepackage{comment}
\usepackage[latin1]{inputenc} 
\usepackage{tikz}
\usetikzlibrary{shapes,arrows}
\usepackage{cite}
\usepackage{url}

\makeatletter
\newcommand{\dotminus}{\mathbin{\text{\@dotminus}}}

\newcommand{\@dotminus}{%
  \ooalign{\hidewidth\raise1ex\hbox{.}\hidewidth\cr$\m@th-$\cr}%
}
\makeatother

\usepackage{setspace,a4wide,color,xcolor,graphicx}

\def\showauthornotes{1}

\ifnum\showauthornotes=1
\newcommand{\Authornote}[2]{{\sf\small\color{red}{[#1: #2]}}}
\else
\newcommand{\Authornote}[2]{}
\fi

\newtheorem{theorem}{Theorem}[section]
\newtheorem{lemma}[theorem]{Lemma}

\newtheorem{proposition}[theorem]{Proposition}

\newtheorem{observation}[theorem]{Observation}

\newtheorem{claim}[theorem]{Claim}

\theoremstyle{definition}

\newtheorem{definition}[theorem]{Definition}

\def\e{\epsilon}

\newcommand{\VC}{\textnormal{VC}}
\newcommand{\SVC}{\textnormal{SVC}}

\newcommand{\calC}{\mathcal{C}}

\newcommand{\calA}{\mathcal{A}}
\newcommand{\calB}{\mathcal{B}}
\newcommand{\calD}{\mathcal{D}}
\newcommand{\calP}{\mathcal{P}}

\newcommand{\calE}{\mathcal{E}}
\newcommand{\calG}{\mathcal{G}}

\newcommand{\calQ}{\mathcal{Q}}

\newcommand{\calS}{\mathcal{S}}
\newcommand{\calT}{\mathcal{T}}

\newcommand{\calV}{\mathcal{V}}

\newcommand{\calX}{\mathcal{X}}

\def\triads{\operatorname{Triads}}

\begin{document}
\title[]{Growth of regular partitions 1: Improved bounds for small slicewise VC-dimension}
\author{C. Terry}\thanks{The author was partially supported by NSF CAREER Award DMS-2115518 and a Sloan Research Fellowship}
\address{Department of Mathematics, The Ohio State University, Columbus, OH 43210, USA}
\email{terry.376@osu.edu}
\maketitle

\begin{abstract}
This is Part 1 in a series of papers about sizes of regular partitions of $3$-uniform hypergraphs.  Previous work of the author and Wolf, and independently Chernikov and Towsner, showed that $3$-uniform hypergraphs of small slicewise VC-dimension admit homogeneous partitions. The results of Chernikov and Towsner did not produce explicit bounds, while the work of the author and Wolf relied on a strong version of hypergraph regularity, and consequently produced a Wowzer type bound on the size of the partition.  This paper gives a new proof of this result, yielding $\e$-homogeneous partitions of size at most $2^{2^{\e^{-K}}}$, where $K$ is a constant depending on the slicewise VC-dimension.  This result is a crucial ingredient in Part 2 of the series, which investigates the growth of weak regular partitions in hereditary properties of $3$-uniform hypergraphs.  
\end{abstract}

\section{Introduction}

The notion of VC-dimension is an important measure of complexity in computer science, graph theory, and logic.  The \emph{VC-dimension} of a graph $G=(V,E)$, denoted $\VC(G)$, is the maximum $k$ for which there exist vertices $\{a_1,\ldots, a_k\}\cup \{b_S:S\subseteq [k]\}$ so that $a_ib_S\in E$ if and only if $i\in S$.  In there words, $VC(G)$ is the VC-dimension of the set system $(V,\calS)$ where 
$$
\calS=\{N_G(x): x\in V(G)\}.
$$
Graphs with small VC-dimension satisfy ``better" structure theorems than can be proved in the setting of arbitrary graphs.  Examples of this type of result include theorems related to property testing \cite{Alon.2000, Alon.2007}, bounds in Ramsey's theorem  \cite{erdos2, Fox.2017bfo, NSP}, removal lemmas \cite{GishbolinerShapira}, and regularity lemmas \cite{Fox.2017bfo, Alon.2007, Chernikov.2016zb, Lovasz.2010}.

Szemer\'{e}di's regularity lemma is a general tool in graph theory which says that any graph can be partitioned into a small number of parts, so that most pairs of parts behave quasirandomly (these are called the \emph{regular pairs}).   There is a parameter $\e$ which measures the quasirandomness of the regular pairs, and a bound $M$, depending on $\e$, on the number of parts needed.  It was shown by Gowers \cite{Gowers.1997} that in general, the dependence of $M$ on $\e$ is at least a tower of exponentials of height polynomial in $\e^{-1}$.  Fox and Lov\'{a}sz later obtained tight bounds for $M$ as a tower of exponentials of height about $\e^{-2}$ (see \cite{Fox.2014}).

In work related to property testing,  Alon, Fischer, and Newman showed that in bipartite graphs of small VC-dimension, an improved regularity lemma holds, where the regular pairs all have density close to $0$ or $1$, and where the number of parts is only polynomial in $\e^{-1}$ \cite{Alon.2007}.  This result was extended to all graphs with small VC-dimension by Lov\'{a}sz-Szegedy \cite{Lovasz.2010}. A quantitative improvement was given by Fox-Pach-Suk in \cite{Fox.2017bfo}, which we will state below after some definitions.  First, given a graph $G=(V,E)$ and sets $X,Y\subseteq V$, we say the pair $(X,Y)$ is \emph{$\e$-homogeneous with respect to $G$} if 
$$
d_G(X,Y):=\frac{|\{(x,y)\in X\times Y: \{x,y\}\in E\}|}{|X||Y|}\in [0,\e)\cup (1-\e,1].
$$
We then say a partition $\calP$ of $V$ is \emph{$\e$-homogeneous with respect to $G$} if at least $(1-\e)|V|^2$ pairs from $V^2$ come from an $\e$-homogeneous pair from $\calP^2$ (see Definition \ref{def:homgraph}).

\begin{theorem}[Fox, Pach, Suk \cite{Fox.2017bfo}]\label{thm:foxpachsuk}
For all $k\geq 1$ there is $c>0$ so that the following holds.  For all $\e\in (0,1/4)$, any graph with VC-dimension at most $k$ has an $\e$-homogeneous equipartition of size $t$ for some $8\e^{-1}\leq t\leq c\e^{-2k-1}$.
\end{theorem}

In fact,  Fox, Pach and Suk proved a much more general theorem which applies to $k$-uniform hypergraphs.  For a $k$-uniform hypergraph $H=(V,E)$, they define the \emph{VC-dimension of $H$}, denoted $\VC(H)$, to be the VC-dimension of the set system $(V,\calS)$ where 
$$
\calS=\{N_H(v_1,\dots, v_{k-1}): v_1,\ldots, v_{k-1}\in V\}.
$$
One can define homogeneous partitions for $k$-uniform hypergraphs in much the same way as above.  In particular, if $G=(V,E)$ is a $k$-uniform hypergraph and $\calP$ is a partition of $V$, then a $k$-tuple $(X_1,\ldots, X_k)\in \calP^k$ is \emph{$\e$-homogeneous with respect to $G$} if 
$$
d_G(X_1,\ldots, X_k):=\frac{|\{(x_1,\ldots, x_k)\in X_1\times \ldots \times X_k: \{x_1,\ldots, x_k\}\in E\}|}{|X_1|\cdots|X_k|}\in [0,\e)\cup (1-\e,1].
$$
We say then $\calP$ is \emph{$\e$-homogeneous with respect to $G$} if at least $(1-\e)|V|^k$ of the tuples from $V^k$ come from an $\e$-homogeneous tuple from $\calP^k$.  Fox, Pach, and Suk proved the following extension of Theorem \ref{thm:foxpachsuk} in terms of the \emph{dual VC-dimension} of a hypergraph.\footnote{ For the purposes of this introduction, it suffices to know the dual VC-dimension is bounded in terms of the VC-dimension. We refer the reader to   \cite{Fox.2014} for more details.}

\begin{theorem}\label{thm:foxpachsuk2}
For all $d,k\geq 1$, there is a constant $c$ so that the following holds.  Suppose $H=(V,E)$ is an $k$-uniform hypergraph with dual VC-dimension at most $d$.  Then it has an $\e$-homogeneous equipartition of size $\e^{-1}\leq t\leq c\e^{-2d-1}$.
\end{theorem}

Similar results with a weaker bound were independently proved by Chernikov and Starchenko in \cite{Chernikov.2016zb}.  The work in \cite{Chernikov.2016zb} is part a line of inquiry within model theoretic combinatorics seeking to understand analogues of VC-dimension in the setting of $k$-uniform hypergraphs for $k>2$ (see e.g. \cite{Chernikov.2019b, Chernikov.2020, Terry.2021a, Terry.2021b, Terry.2022, MC,Hempel.2014}).  Two papers in this vein \cite{Terry.2021b, Chernikov.2020} show that there is a larger class of $3$-uniform hypergraphs which admit homogeneous partitions, beyond just those of bounded VC-dimension.   Stating these results requires a different  analogue of VC-dimension for $3$-uniform hypergraphs which involves the following notion of ``slice graphs."  Given a $3$-uniform hypergraph $H=(V,F)$ and $x\in V$, the \emph{slice graph of $H$ at $x$} is the graph
$$
H_x=(V,\{yz\in {V\choose 2}: xyz\in F\}).
$$
We then define the \emph{slicewise VC-dimension\footnote{In \cite{Terry.2021b}  this is referred to as \emph{weak VC-dimension}.} of $H$} to be $\SVC(H):=\max\{VC(H_x): x\in V\}$.

It is an exercise to check that for any $3$-uniform hypergraph $H$, $\SVC(H)\leq \VC(H)$.  Work of Wolf and the author showed the existence of homogeneous partitions can be extended to the more general assumption of having bounded slicewise VC-dimension\cite{Terry.2021b}.  Non-quantitative versions of this result (as well as analogues for $k$-uniform hypergraphs) were independently proved by Towsner and Chernikov in \cite{Chernikov.2020}.

\begin{theorem}[Chernikov-Towsner \cite{Chernikov.2020}, Terry-Wolf \cite{Terry.2021b}]\label{thm:slvcwow}
For all $k\geq 1$ and $\e>0$, there is an integer $M(\e)$ so that any $3$-graph with slicewise VC-dimension less than $k$ admits an $\e$-homogeneous partition with at most $M(\e)$ parts.
\end{theorem}

While the proof of the author and Wolf of Theorem \ref{thm:slvcwow} is quantitative, it relies on a strong version of the hypergraph regularity lemma, and thus yields a Wowzer style upper bound on the partition itself.  The main result of this paper is the following quantitative improvement.

\begin{theorem}\label{thm:slvccor}
For all $k\geq 1$ there are $K,\e^*>0$ so that the following holds. Suppose $0<\e<\e^*$ and $H$ is a sufficiently large $3$-uniform hypergraph with slicewise VC-dimension less than $k$. Then there exists an $\e$-homogeneous partition for $H$ of size at most $2^{2^{\e^{-K}}}$.
\end{theorem}

Theorem \ref{thm:slvccor} answers a question posed in \cite{Terry.2021b}.  This theorem is also crucial for Part 2 of this series, which considers the possible growth rates of weak regular partitions in hereditary properties of $3$-uniform hypergraphs.  

The biggest open problem which remains from this paper and its sequels is whether the bound in Theorem \ref{thm:slvccor} can be improved.  It is shown in Part 2 that there exist arbitrarily large $3$-uniform hypergraphs with slicewise VC-dimension $2$, and which require at least $2^{\e^{-C}}$ parts in any $\e$-homogeneous partition.  Thus, if the bound in Theorem \ref{thm:slvccor} can be improved, it will not be beyond a single exponential in $\e^{-1}$.  The author conjectures the form of the bound in Theorem \ref{thm:slvccor} is tight.  We refer the reader to Part 2 for more discussion of this question.

\subsection{Acknowledgement} The author would like to thank Julia Wolf for helpful conversations regarding this result.  The author would also like to thank Henry Towsner for discussions about this kind of result.  Forthcoming work of the author and Towsner will address generalizations of these results to higher arities.

\subsection{Notation}\label{ss:notation}
For a set $V$ and an integer $k\geq 1$, define ${V\choose k}=\{X\subseteq V: |X|=k\}$.  A \emph{$k$-uniform hypergraph} is a pair $(V,E)$ where $E\subseteq {V\choose k}$.  Given a $k$-uniform hypergraph $G$, $V(G)$ will denote its vertex set of $V$ and $E(G)$ will denote its edge set of $G$. A \emph{graph} is a $2$-uniform hypergraphs as \emph{graphs}. For $k\geq 3$, a \emph{$k$-graph} is a $k$-uniform hypergraph.  Given a $k$-graph $(V,E)$, an induced sub-$k$-graph has the form $(V',E')$ where $V'\subseteq V$ and $E'=E\cap {V'\choose 2}$.   We let $G[V']$ be the induced sub-$k$-graph of $G$ with vertex set $V'$.

We will use the notation $xy$ for a two element set $\{x,y\}$, and $xyz$ for a three element set $\{x,y,z\}$.  For sets $X,Y,Z$, define
\begin{align*}
K_2[X,Y]&=\{xy: x\in X, y\in Y, x\neq y\}\text{ and }\\
K_3[X,Y,Z]&=\{xyz: x\in X, y\in Y, z\in Z, x\neq y, y\neq z, x\neq z\}.
\end{align*}
Suppose $G=(V,E)$ is a graph.  We let $\overline{E}=\{(x,y)\in V^2: xy\in E\}$.  Given sets $X,Y\subseteq V$, the \emph{density of $(X,Y)$ in $G$} is 
$$
d_G(X,Y):=|\overline{E}\cap (X\times Y)|/|X||Y|.
$$ 
For any $v\in V$, the \emph{neighborhood of $v$ in $G$} is $N_G(v)=\{x\in X: xy\in E\}$.  Given a subset $E'\subseteq {V\choose 2}$, we will also write 
\begin{align}\label{nbd}
N_{E'}(v)=\{x\in X: xy\in E\}.
\end{align} 

Suppose now $H=(V,E)$ is a $3$-uniform hypergraph.    We let 
$$
\overline{E}=\{(x,y,z)\in V^3: xyz\in E\}.
$$
For disjoint sets $X,Y,Z\subseteq V$, define
$$
d_H(X,Y,Z):=|\overline{E}\cap (X\times Y\times Z)|/|X||Y||Z|.
$$
For any $x,y\in V$, we let
$$
N_H(x)=\{uv\in {V\choose 2}: xuv\in E\}\text{ and }N_H(xy)=\{z\in V: xyz\in E\}.
$$
We will write $E^1=E$ and $E^0={V\choose k}\setminus E$.  More generally, given any set $E'\subseteq{V\choose 3}$, we will also write
$$
N_{E'}(x)=\{uv\in {V\choose 2}: xuv\in E'\}\text{ and }N_{E'}(xy)=\{z\in V: xyz\in E'\}.
$$

Given integers $k\leq \ell$, We say a $k$-graph $G=(V,E)$ is \emph{$\ell$-partite} if there is a partition $V=V_1\cup \ldots \cup V_{\ell}$ so that for every $e\in E$, there are $1<i_1<\ldots<i_k\leq \ell$ so that for every $e\in E$ and $1\leq i\leq \ell$, $|e\cap V_i|\leq 1$.  We will write $G=(V_1\cup \ldots \cup V_{\ell}, E)$ to denote that $G$ is $\ell$-partite with the displayed partition of vertices.
 
Given $r_1,r_2\in \mathbb{R}$ and $\e>0$, we use the notation $r_1=r_2\pm \e$ or $r_1\approx_{\e}r_2$ to mean that $r_1\in (r_2-\e,r_2+\e)$. For all integers $n\geq 1$, $[n]=\{1,\ldots, n\}$.  An \emph{equipartition} of a set $V$ is a partition $V=V_1\cup \ldots \cup V_t$ satisfying $||V_i|-|V_j||\leq 1$ for all $1\leq i,j\leq t$.  We will require  the following growth functions.

Given a set $V$, a partition $\calP$ of $V$ and a subset $X\subseteq V$, let $\calP|_X:=\{X\cap S: S\in \calP\}$. Note that in this case, $\calP|_X$ is a partition of $X$.

 \section{Preliminaries}
 
 This section contains some required preliminaries.  These will include lemmas about averaging and symmetry in $3$-graphs (Section \ref{ss:sym}),  definitions and results about VC-dimension and homogeneity (Section \ref{ss:vc}), and finally,  tripartite decompositions and related notions (Section \ref{ss:dec}).
 
 \subsection{Averaging and symmetry}\label{ss:sym}
 
 This section contains two crucial lemmas.  The first is a simple averaging lemma.

\begin{lemma}\label{lem:averaging}
Let $a,b,\e\in (0,1)$ satisfy $ab=\e$. Suppose $A\subseteq X$ and $|A|\geq (1-\e)|X|$.  For any partition $\calP$ of $X$, if we let $\Sigma=\{Y\in \calP: |A\cap Y|\geq (1-a)|Y|\}$, then $|\bigcup_{Y\in \Sigma}Y|\geq (1-b)|X|$.
\end{lemma}
\begin{proof}
Suppose towards a contradiction that $|\bigcup_{Y\in \Sigma}Y|< (1-b)|X|$. Then 
$$
|X\setminus A|\geq \sum_{Y\in \calP\setminus \Sigma}|Y\setminus A|>\sum_{Y\in \calP\setminus \Sigma}a|Y|=a\sum_{Y\in \calP\setminus \Sigma}|Y|\geq ab|X|=\e |X|,
$$
 a contradiction.
\end{proof}

The second important lemma is a ``symmetry lemma" for $3$-partite $3$-graphs  We begin by stating Lemma \ref{lem:twosticks} below, which is a simpler symmetry lemma in the setting of graphs.  This will involve an adaptation of the notion of ``good sets", first defined in \cite{Malliaris.2014}. 

\begin{definition}\label{def:almostgoodgraph}
Suppose $G=(A\cup B, E)$ is a bipartite graph and $\e>0$.  We say a subset $X\subseteq B$ is \emph{almost $\e$-good with respect to $G$} if there is a set $A'\subseteq A$ so that $|A'|\geq (1-\e)|A|$ and for all $a\in A'$, 
$$
\frac{|N_G(a)\cap X|}{|X|}\in [0,\e)\cup (1-\e,1].
$$
\end{definition}

 The symmetry lemma for graphs,   Lemma \ref{lem:twosticks} below, says that if $G=(A\cup B, E)$ is a bipartite graph, and both $A$ and $B$ are almost $\e$-good with respect to $G$, then $d_G(A,B)$ is close to $0$ or $1$.   A proof appears in \cite{Terry.2021b} (see Lemma 4.9 there).

\begin{lemma}[Symmetry Lemma for Graphs]\label{lem:twosticks}
For all $0<\e<1/4$ there is $n$ such that the following holds.  Suppose $G=(U\cup W, E)$ is a bipartite graph, $|U|,|W|\geq n$, and both $U$ and $W$ are almost $\e$-good sets in $G$.  Then $d_G(U,W)\in [0,2\e^{1/2})\cup (1-2\e^{1/2},1]$. 
\end{lemma}
Our goal is to prove a ternary analogoue of Lemma \ref{lem:twosticks}. We begin with an analogue of Definition \ref{def:almostgoodgraph}.

\begin{definition}\label{def:almostgood}
Suppose $H=(V_1\cup V_2\cup V_3,E)$ is a $3$-partite $3$-graph and  $\{i_1,i_2,i_3\}=\{1,2,3\}$.  A set $X\subseteq V_{i_1}$ is \emph{almost $\e$-good with respect to $H$} if for at least $(1-\e)|V_{i_2}||V_{i_3}|$ many $(y,z)\in V_{i_2}\times V_{i_3}$ satisfy $|N_H(yz)\cap X|/|X|\in [0,\e)\cup (1-\e,1]$.
\end{definition}

We now state and prove our symmetry lemma for $3$-graphs, which says that a tripartite graph in which all three parts are almost $\e$-good has density close to $0$ or $1$.

\begin{lemma}[Symmetry Lemma for $3$-graphs]\label{lem:symmetry}
Let $0<\e<1/100$ and suppose $H=(V_1\cup V_2\cup V_3, E)$ is a tripartite $3$-graph and assume $V_1,V_2,V_3$ are each almost $\e$-good with respect to $H$.  Then $d_H(V_1,V_2,V_3)\in [0,4\e^{1/16})\cup (1-4\e^{1/16},1]$.
\end{lemma}
\begin{proof}
Given $i,j,k$ satisfying $\{i,j,k\}=\{1,2,3\}$, define
$$
E_{ij}=\{xy\in K_2[V_i,V_j]: |N_H(xy)|\leq \e|V_k|\text{ or }|N_H(xy)|\geq (1-\e)|V_k|\}.
$$
Our assumption implies that  for each $ij\in {[3]\choose 2}$, $|E_{ij}|\geq (1-\e)|V_i||V_j|$.  Consider now the bipartite graph $\Gamma$ with vertex set $V_1\cup K_2[V_2,V_3]$ and edge set 
$$
E(\Gamma)=\{(x, yz)\in V_1\cup K_2[V_2,V_3]: xyz\in H\}.
$$
It suffices to show $d_{\Gamma}(V_1,K_2[V_2,V_3])\in [0,2\e^{1/16})\cup (1-2\e^{1/16},1]$.  We will show this by applying Lemma \ref{lem:twosticks} to $\Gamma$. To do this, we need to verify the relevant hypotheses.  

Since $|E_{23}|\geq (1-\e)|K_2[V_2,V_3]|$, we already know $K_2[V_2,V_3]$ is almost $\e$-good with respect to $\Gamma$.  Our next goal is to show  $V_1$ is almost $2\e^{1/4}$-good with respect to $\Gamma$.  We begin by defining the following subsets of $E_{12}$.
\begin{align*}
E_{12}^1&=\{xy \in K_2[V_1,V_2]: |N_H(xy)|\geq (1-\e)|V_3|  \}\text{ and } \\
E_{12}^0  &=\{xy \in K_2[V_1,V_2]: |N_H(xy)|\leq \e|V_3| \}.
\end{align*}
Note $E_{12}=E_{12}^1\sqcup E_{12}^0$.  Consider the following subset of $V_1$ (recall notation (\ref{nbd})).
$$
V_1'=\{x\in V_1: |N_{E_{12}}(x)|\geq (1-\e^{1/2})|V_2|\text{ and }|N_{E_{13}}(x)|\geq (1-\e^{1/2})|V_3|\}.
$$
Since $|E_{12}|\geq (1-\e)|V_3|$ and $|E_{13}|\geq (1-\e)|V_2|$, $|V_1'|\geq (1-2\e^{1/2})|V_1|$. 

\begin{claim}\label{cl:sym1}
For each $x\in V_1'$, either $|N_{E_{12}^1}(x)|\geq (1-\e^{1/4})|V_2|$ or $|N_{E_{12}^0}(x)|\geq (1-\e^{1/4})|V_2|$. 
\end{claim}
\begin{proof} Suppose towards a contradiction there is some $x\in V_1'$ so that
$$
\min\{|N_{E^1_{12}}(x)|,|N_{E^0_{12}}(x)|\}\geq \e^{1/4}|V_2|.
$$
To ease notation, set
\begin{align*}
V_2^1=N_{E^1_{12}}(x)\text{ and }V_2^0=N_{E^0_{12}}(x).
\end{align*}
Then $V_2^1\cup V_2^0\subseteq V_2$ and $\min\{|V_2^1|,|V_2^0|\}\geq \e^{1/4}|V_2|$.  Let $H_x$ be the slice graph of $H$ at $x$, i.e. the graph with vertex set $V(H)$ and edge set $\{yz\in {V\choose 2}: xyz\in E(H)\}$.  By assumption, the following hold.
\begin{align}
|E(H_x)\cap K_2[V_2^1, V_3]|&\geq \sum_{y\in V_2^1}|N_H(xy)|\geq (1-\e)|V^1_2||V_3|\text{ and }\label{al:e1}\\
|K_2[V_2^0, V_3]\setminus E(H_x)|&\geq \sum_{y\in V_2^0}|V_3\setminus N_H(xy)|\geq (1-\e)|V^0_2||V_3|.\label{al:e0}
\end{align}
Define
\begin{align*}
V_3^1&=\{z\in V_3: |N_H(xz)\cap V_2^1|\geq (1-\e^{1/2})|V_2^1|\}\text{ and }\\
V_3^0&=\{z\in V_3: | V_2^0\setminus N_H(xz)|\geq (1-\e^{1/2})|V_2^0|\}.
\end{align*}
By (\ref{al:e1}) and (\ref{al:e0}), we have $|V_3^1|\geq (1-\e^{1/2})|V_3|$ and $|V_3^0|\geq (1-\e^{1/2})|V_3|$ and 
consequently, 
$$
|V_3^1\cap V_3^0\cap N_{E_{13}}(x)|\geq \e^{1/4}|V_3|-2\e^{1/2}|V_3|>0,
$$
where the last inequality is because $\e<1/100$. From this, we conclude there exists some $z\in N_{E_{13}}(x)\cap V_3^1\cap V_3^0$.  Since $z\in V_3^1$, 
$$
|N_H(xz)|\geq (1-\e^{1/2})|V_2^1|\geq (1-\e^{1/2})\e^{1/4}|V_2|>\e|V_2|.
$$
On the other hand, since $z\in V_3^0$, 
$$
|V_2\setminus N_H(xz)|\geq (1-\e^{1/4})|V_2^0|\geq (1-\e^{1/4})\e^{1/4}|V_2|>\e|V_2|.
$$
But this contradicts that $xz\in E_{13}$.   This finishes the proof of Claim \ref{cl:sym1}.
\end{proof}

As a consequence of Claim \ref{cl:sym1}, we have that for every $x\in V_1'$, either $|N_{E_{12}^0}(x)|\leq \e^{1/4}|V_2|$, or $|N_{E_{12}^1}(x)|\leq \e^{1/4}|V_2|$.  Given $x\in V_1'$, if $|N_{E_{12}^0}(x)|\leq \e^{1/4}|V_2|$, then
\begin{align*}
|N_H(x)\cap K_2[V_3,V_3]|\geq \sum_{y\in N_{E_{12}^1}(x)}|N_H(xy)|\geq |N_{E_{12}^1}(x)|(1-\e)|V_3|&\geq  (1-\e^{1/4})(1-\e)|V_2||V_3|\\
&>(1-2\e^{1/4})|V_2||V_3|.
\end{align*}
If $x\in V_1'$ and $|N_{E_{12}^1}(x)|\leq \e^{1/4}|V_2|$, then a similar argument shows 
$$
|K_2[V_2,V_3]\setminus N_H(x)|>(1-2\e^{1/4})|V_2||V_3|.
$$
We now have that $|V_1'|\geq (1-2\e^{1/2})$, and for all $x\in V_1'$, either $|N_{\Gamma}(x)|>(1-2\e^{1/4})|V_2||V_3|$ or $|N_{\Gamma}(x)|<2\e^{1/4}|V_2||V_3|$. In other words, $V_1$ is almost $2\e^{1/4}$-good in $\Gamma$.  We can thus conclude, by Lemma \ref{lem:twosticks}, that $d_{\Gamma}(V_1,K_2[V_2,V_3])\in [0,4\e^{1/16})\cup (1-4\e^{1/16},1]$.  This finishes the proof.
\end{proof}

 \subsection{VC-dimension and Homogeneous Partitions of Graphs}\label{ss:vc}
 
 In this section we discuss VC-dimension and homogeneous partitions for graphs.  We begin with the definition of VC-dimension of a graph. 
 
\begin{definition}\label{def:vc}
The VC-dimension of a graph $G=(V,E)$, denote $\VC(G)$, is the largest $k$ so that there exist vertices $\{a_i: i\in [k]\}\cup \{b_S:S\subseteq [k]\}\subseteq V$, so that for all $i\in [k]$ and $S\subseteq[k]$, $a_ib_S\in E$ if and only if $i\in S$.
\end{definition}

We now define homogeneous partitions for graphs.
 
\begin{definition}\label{def:homgraph}
Suppose $G=(V,E)$ is an graph. 
\begin{enumerate}
\item For subsets $X,Y\subseteq V$, the pair $(X,Y)$ is \emph{$\e$-homogeneous with respect to $G$} if $d_G(X,Y)\in [0,\e)\cup (1-\e,1]$.
\item We say a partition $\calP$ of $V$ is \emph{$\e$-homogeneous with respect to $G$} if for at least $(1-\e)|V|^2$ of the tuples $(x,y)\in V^2$, there is some $\e$-homogeneous pair $(X,Y)\in \calP^2$ so that $(x,y)\in X\times Y$. 
\end{enumerate}
\end{definition}

As mentioned in the introduction, work of Alon-Fischer-Newman, Lov\'{a}sz-Szegedy, and Fox-Pach-Suk have shown that graphs of small VC-dimension admit homogeneous partitions with efficient bounds (see Theorem \ref{thm:foxpachsuk} in the introduction).  We will require a slightly more complicated version of Theorem \ref{thm:foxpachsuk}, which allows us to find homogeneous \emph{equipartitions} of bipartite graphs which also refine a distinguished bipartition.  The proof follows from Theorem \ref{thm:foxpachsuk} via standard arguments (which we include in the appendix for completeness). 

\begin{theorem}\label{thm:bipfoxpachsuk}
For all $k\geq 1$ there exist constants $C=C(k)$ and $D=2k+2$  so that for all $\e\in (0,1/4)$ the following holds.  Suppose $H=(A\cup B, E)$ is a bipartite graph with VC-dimension less than $k$ and $|A|=|B|$.  Then there exists an equipartition $\calP$ of $A\cup B$ refining $\{A,B\}$ which is $\e$-homogeneous for $H$ and so that $\e^{-1}\leq |\calP|\leq C\e^{-D}$.
\end{theorem}

 \subsection{Slicewise VC-dimension and Homogeneous Partitions of $3$-Graphs}\label{ss:dec}
 
We begin with the definition of a ``slice graph" associated to a $3$-graph.
  
 \begin{definition}\label{def:slicegraph}
Suppose $H$ is a $3$-graph.  For each $x\in V(H)$, the \emph{slice graph of $H$ at $x$}, denoted $H_x$, is the graph with vertex set $V(H_x)=V(H)$ and edge set 
$$
E(H_x)=\{yz\in {V(H)\choose 2}: xyz\in E(H)\}.
$$
\end{definition} 

We now define the slicewise VC-dimension of a $3$-graph to be the maximum VC-dimension of its associated slice graphs.

\begin{definition}
Suppose $H=(V,E)$ is a $3$-graph.  The \emph{slicewise VC-dimension of $H$} is
$$
\SVC(H):=\max\{VC(H_x): x\in V\}.
$$

\end{definition}

We now define homogeneous partitions for $3$-graphs in analogy to Definition \ref{def:homgraph}.
 
 \begin{definition}\label{def:hom}
Suppose $G=(V,E)$ is an $3$-graph.
\begin{enumerate}
\item Given subsets $X,Y,Z\subseteq V$, we say the triple $(X,Y,Z)$ is \emph{$\e$-homogeneous} if $d_G(X,Y,Z)\in [0,\e)\cup (1-\e,1]$.
\item We say a partition $\calP$ of $V$ is \emph{$\e$-homogeneous with respect to $G$} if for at least $(1-\e)|V|^3$ many $(x,y,z)\in V^3$, there is an $\e$-homogeneous triple $(X,Y,Z)\in \calP^3$ so that $(x,y,z)\in X\times Y\times Z$. 
\end{enumerate}
\end{definition}

It is known that $3$-graphs of small slicewise VC-dimension admit homogeneous partitions (see Theorem \ref{thm:slvcwow} in the introduction).  The goal of this paper is to prove they admit homogeneous partitions with a relatively small number of parts.  A crucial tool in this proof will be the notion of an almost $\e$-good partition, which we now define (recall Definition \ref{def:almostgood} as well).

\begin{definition}\label{def:goodpart}
Suppose $H=(V_1\cup V_2\cup V_3,E)$ is a $3$-partite $3$-graph and $\{i_1,i_2,i_3\}=\{1,2,3\}$.  A partition $\calP$ of $V_{i_1}$ is \emph{almost $\e$-good with respect to $H$} if there is a set $\calP'\subseteq \calP$ so that $|\bigcup_{X\in \calP'}X|\geq (1-\e)|V_{i_1}|$ and every element in $\calP$ is almost $\e$-good with respect to $H$.  
\end{definition}

We now extend Definition \ref{def:goodpart} to partitions of the whole vertex set of a tripartite $3$-graph.

\begin{definition}
We say a partition $\calP$ of $V(H)$ refining $\{V_1,V_2,V_3\}$ is \emph{almost $\e$-good with respect to $H$} if $\calP|_{V_1}$, $\calP|_{V_2}$ and $\calP|_{V_3}$ are almost $\e$-good with respect to $H$.
\end{definition}

We now show that an almost $\e$-good partition of a tripartite $3$-graph is automatically homogeneous.  This relies crucially on our symmetry lemma, Lemma \ref{lem:symmetry}.

\begin{proposition}\label{prop:goodhom}
Suppose $H=(V_1\cup V_2\cup V_3,E)$ is a tripartite $3$-graph, and $\calP$ is a partition of $V(H)$ refining $\{V_1,V_2,V_3\}$ which is an almost $\e$-good partition with respect to $H$.  Then $\calP$ is $28\e^{1/64}$-homogeneous with respect to $H$.
\end{proposition}
\begin{proof}
To ease notation, let $\calP_1=\calP|_{V_1}$, $\calP_2=\calP|_{V_2}$, and $\calP_3=\calP|_{V_3}$.  Define
\begin{align*}
\Gamma_{1}&=\bigcup_{Y\in \calP_1}\{xyz\in K_2[V_1,V_2,V_3]: |N_H(yz)\cap Y|/|Y|\in [0,\e)\cup (1-\e,1]\},\\
\Gamma_{2}&=\bigcup_{Y\in \calP_2}\{xyz\in K_2[V_1,V_2,V_3]: |N_H(xz)\cap Y|/|Y|\in [0,\e)\cup (1-\e,1]\},\text{ and }\\
\Gamma_{3}&=\bigcup_{Y\in \calP_3}\{xyz\in K_2[V_1,V_2,V_3]: |N_H(xy)\cap Y|/|Y|\in [0,\e)\cup (1-\e,1]\}.
\end{align*}
By assumption, each of $\Gamma_1$, $\Gamma_2$, and $\Gamma_3$ have size at least $(1-\e)|V_1||V_2||V_3|$. Thus if we set 
$$
\Gamma:=\Gamma_1\cap \Gamma_2\cap \Gamma_3,
$$
then $|\Gamma|\geq (1-3\e)|V_1||V_2||V_3|$. Let 
$$
\Sigma=\{(X,Y,Z)\in \calP_1\times \calP_2\times \calP_3: |K_3[X,Y,Z]\cap \Gamma|\geq (1-\sqrt{3\e})|V_1||V_2||V_3|\}.
$$
  By the lower bound above for $|\Gamma|$ and Lemma \ref{lem:averaging}, 
  $$
  |\bigcup_{(X,Y,Z)\in \Sigma}K_3[X,Y,Z]|\geq (1-\sqrt{3\e})|V_1||V_2||V_3|.
  $$ 
Thus it suffices to show that each $(X,Y,Z)\in \Sigma$ is $28\e^{1/64}$-homogeneous with respect to $H$.  To this end, fix a triple $(X,Y,Z)\in \Sigma$.  Let 
$$
E_{XY}=\{xy\in K_2[X,Y]: |N_{\Gamma}(xy)\cap Z|\geq (1-(3\e)^{1/4})|Z|\}.
$$
Since $(X,Y,Z)\in \Sigma$, $|E_{XY}|\geq (1-(3\e)^{1/4})|X||Y|$.  For all $xy\in E_{XY}$, since $N_{\Gamma}(xy)\cap Z\neq \emptyset$, we have by definition of $\Gamma$ that $|N_H(xy)\cap Z|/|Z|\in [0,\e)\cup (1-\e,1]$.  Combining with the lower bound on $E_{XY}$, we have shown that $Z$ is almost $3\e^{1/4}$-good with respect to $H[X,Y,Z]$. A symmetric argument shows both $X$ and $Y$ are also almost $(3\e)^{1/4}$-good with respect to $H[X,Y,Z]$.  By Lemma \ref{lem:symmetry}, $(X,Y,Z)$ is $28\e^{1/64}$-homogeneous with respect to $H$, as desired.
\end{proof}

\subsection{Tripartite Decompositions}

In this section we define tripartite decompositions. These will play an important role in the proof of Theorem \ref{thm:slvc}.  They are essentially tripartite versions of more general  decompositions typically used in a strong form of regularity for $3$-graphs.  Our notation is largely adapted from \cite{Frankl.2002}.

We begin by introducing the notion of a triad, which for us will be simply any tripartite graph.
\begin{definition}
A \emph{triad} is a tripartite graph.
\end{definition}

Given a triad $G=(X\cup Y\cup Z, E)$, we let $K_3^{(2)}(G)$ denotes the set of triangles from $G$, i.e. we define 
$$
K_3^{(2)}(G)=\{xyz\in K_3[X,Y,Z]: xy, yz, xz\in E\}.
$$
The notion of a density of a triad, relative to a $3$-graph, is crucial for our proofs.

\begin{definition}
Suppose $H$ is a $3$-graph, $X,Y,Z$ are disjoint subsets of $V(H)$, and $G=(X\cup Y\cup Z, E)$ is a triad. Define
$$
d_H(G)=\frac{|E(H)\cap K_3^{(2)}(G)|}{| K_3^{(2)}(G)|}.
$$
\end{definition}

We now define homogeneous triads of a $3$-graph relative to a triad.

\begin{definition}
Suppose $H$ is a $3$-graph, $X,Y,Z$ are disjoint subsets of $V(H)$, and $G=(X\cup Y\cup Z, E)$ is a triad.  We say $G$ is \emph{$\e$-homogeneous with respect to $H$} if 
$$
d_H(G)\in [0,\e)\cup (1-\e,1].
$$
\end{definition}

We now define the notion of a tripartite decomposition.

\begin{definition}\label{def:tridec}
Suppose $V=A\sqcup B\sqcup C$.  A \emph{tripartite $(A,B,C)$-decomposition of $V$} is a tuple $\calP=(\calP_1,\calP_2)$ satisfying the following. 
\begin{itemize}
\item $\calP_1=(\calP_A,\calP_B,\calP_C)$, where $\calP_A,\calP_B,\calP_C$ are partitions of of $A$, $B$, and $C$ respectively,
\item $\calP_2=\{\calP_{XY}: (X,Y)\in (\calP_A\times \calP_B)\cup (\calP_A\times \calP_C)\cup (\calP_B\times \calP_C)\}$, where for each $(X,Y)\in (\calP_A\times \calP_B)\cup (\calP_A\times \calP_C)\cup (\calP_B\times \calP_C)$, $\calP_{XY}$ is a partition of $K_2[X,Y]$.
\end{itemize}
\end{definition}
In Definition \ref{def:tridec}, we will sometimes all $\calP$ simply a \emph{tripartite decomposition of $V$} when the relevant partition $A\sqcup B\sqcup C$ is clear from context.  Tripartite decompositions come equipped with an  associated collection of triads, which we now define.

\begin{definition}\label{def:tripdecomp}
Suppose $V=A\sqcup B\sqcup C$ and $\calP$ is a tripartite $(A,B,C)$-decomposition of a $V$.  
\begin{itemize}
\item A \emph{triad of $\calP$} is a tripartite graph of the form 
$$
(X\cup Y\cup Z, P\cup Q\cup R),
$$
where $(X,Y,Z)\in \calP_A\times \calP_B\times \calP_C$, and $P\in \calP_{XY}$, $Q\in \calP_{XZ}$, and $R\in \calP_{YZ}$.
\item $Tri(\calP)$ denotes the set of triads of $\calP$.
\end{itemize}
\end{definition}

Using the notation of Definition \ref{def:tripdecomp}, we see that a tripartite decomposition $\calP$ induces a partition 
$$
K_3[A,B,C]=\bigcup_{G\in Tri(\calP)}K_3^{(2)}(G).
$$
  We now define homogeneous tripartite decompositions.  

\begin{definition}\label{def:triphom}
Suppose $V=A\sqcup B\sqcup C$, $H=(A\cup B\cup C, E)$ is a tripartite $3$-graph, and $\calP$ is a tripartite $(A,B,C)$-decomposition of a $V$. We say $\calP$ is \emph{$\e$-homogeneous with respect to $H$} if there is a set $\Sigma\subseteq Tri(\calP)$ so that every $G\in \Sigma$ is $\e$-homogeneous with respect to $H$, and
$$
\sum_{G\in \Sigma}|K_3^{(2)}(G)|\geq (1-\e)|A||B||C|.
$$
\end{definition}

We note that there are major differences between a homogeneous tripartite decomposition in the sense of Definition \ref{def:triphom} and a homogeneous partition in the sense of Definition \ref{def:hom}. In particular, Definition \ref{def:hom} deals only with triples of sets from a vertex partition, while Definition \ref{def:triphom} deals with tripartite graphs built from partitions of both vertices and pairs of vertices.  

\section{A structure theorem for edge-colored bipartite graphs}\label{ss:ecg}

This section contains a required result about bipartite edge-colored graphs.  As motivation, we begin by discussing one way in which these objects will show up in our proof of Theorem \ref{thm:slvc}.  

At some point in the proof of Theorem \ref{thm:slvccor}, we will apply the Theorem \ref{thm:bipfoxpachsuk} to a bipartite graph $G=(A\cup B,E)$ with VC-dimension less than $k$. This will yield an $\e$-homogeneous equipartition of the form $A_1\cup \ldots \cup A_t\cup B_1\cup \ldots \cup B_t$.  Such a decomposition naturally gives rise to a ``reduced graph" with vertex set $\{A_1,\ldots, A_t,B_1,\ldots, B_t\}$, and with three types of edges: the pairs $A_iB_j$ where $d_G(A_i,B_j)\geq 1-\e$, the pairs $A_iB_j$ where $d_G(A_i,B_j)\leq \e$, and the pairs $A_iB_j$ which are not $\e$-homogeneous with respect to $G$.   Making this precise is part of the motivation for the following definition. 

\begin{definition}\label{def:ecg}
 A \emph{bipartite edge-colored graph} is a tuple $G=(A\cup B,E_0,E_1,\ldots, E_r)$, where $r\geq 1$ and $K_2[A,B]=E_0\sqcup \ldots\sqcup E_r$.  We refer to elements as $A\cup B$ as \emph{vertices}, and the sets $E_0,\ldots, E_r$ as \emph{edge colors}.  
 \end{definition}
 Given a bipartite edge-colored graph $G=(A\cup B,E_0,E_1,\ldots, E_r)$, we will use similar notation as that laid out for graphs in Section \ref{ss:notation}.  In particular, for $0\leq u\leq r$ and $x\in A\cup B$, the \emph{$E_u$-neighborhood of $x$} is the set
$$
N_{E_u}(x):=\{y\in A\cup B: xy\in E_u\}.
$$  
Given subsets $X\subseteq A$ and $Y\subseteq B$, we let $E_u[X,Y]=E_u\cap K_2[X,Y]$.  We then let $G[X, Y]$ be the edge-colored graph with vertex set $X\cup Y$, and edge colors $E_u[X,Y]$ for $0\leq u\leq r$.

As discussed above, one way bipartite edge-colored graphs will arise in our proof is as reduced graphs associated to homogeneous partitions of bipartite graphs of small VC-dimension.  In this case, there will be $3$ edge colors:  $E_0$ (the ``sparse" color), $E_1$ (the ``dense" color), and $E_2$ (the ``error" color).  In this context, the resulting edge-colored graph will omit certain special bipartite configurations.  The bipartite configuration we are interested in is the following.

\begin{definition}
Given $k\geq 1$, define $U(k)$ to be the bipartite graph 
$$
(\{a_i: i\in [k]\}\cup \{b_S:S\subseteq [k]\},\{a_ib_S: i\in S\}).
$$
\end{definition}

We now give a definition of what it means to find a copy of $U(k)$ inside an the kinds of edge-colored graphs we are interested in.

\begin{definition}\label{def:copies}
Let $G=(A\cup B,E_0,\ldots, E_r)$ be a bipartite edge-colored graph.  Given $0\leq i\neq j\leq r$, an \emph{$E_u/E_v$-copy of $U(k)$ with right side in $B$} consists of vertices $\{x_1,\ldots, x_k\}\subseteq A$ and $\{y_S\subseteq [k]\}\subseteq B$ so that for each $1\leq i\leq k$ and $S\subseteq [k]$, if $i\in S$ then $x_iy_S\in E_u$ and if $i\notin S$, then $x_iy_S\in E_v$.

A \emph{$E_u/E_v$-copy of $U(k)$ in $G$} is a copy of $U(k)$ in $G$ with right side in $A$ or with right side in $B$. 
\end{definition}

If $G=(A\cup B,E_0,E_1,E_2)$ is a bipartite edge-colored graph containing no $E_0/E_1$-copies of $U(k)$ and with $E_2$ is small in size, then one can prove a structure theorem for $G$ (such hypotheses will hold when $G$ arises in the way discussed above).  Our next goal is to state such a structure theorem. For this we will need the following definition.

\begin{definition}\label{def:echom}
Let $\e>0$ and suppose $G=(A\cup B, E_0,E_1,E_2)$ is a bipartite edge-colored graph.  Suppose $\calP=(\calP_A,\calP_B)$ where $\calP_A$ is a partition of $A$ and $\calP_B$ is a partition of $B$.  We say $\calP$ is \emph{$\e$-homogeneous for $G$ with respect to $\{E_0,E_1\}$} if there are sets $\Sigma_0,\Sigma_1\subseteq \calP_A\times \calP_B$ so that the following hold.
\begin{enumerate}
\item $|\bigcup_{(X,Y)\in \Sigma_0\cup \Sigma_1}K_2[X,Y]|\geq (1-\e)|A||B|$,
\item For each $\alpha\in \{0,1\}$ and $(X,Y)\in \Sigma_{\alpha}$, $|E_{\alpha}\cap K_2[X,Y]|\geq (1-\e)|X||Y|$.
\end{enumerate}
The \emph{size of $\calP$} is $\max\{|\calP_A|,|\calP_B|\}$.
\end{definition}

We now state the structure theorem for edge-colored graphs which will need in our proof of Theorem \ref{thm:slvccor}.

\begin{proposition}\label{prop:vcremoval}
For all $k\geq 1$, there is $c=c(k)$ so that the following holds. Suppose $0<\e,\delta<1/2$ satisfy $\e\leq c^{-4}(\delta/8)^{4k+4}/8$.    Suppose $G=(A\cup B, E_0,E_1,E_2)$ is a bipartite edge colored graph where $|E_2|\leq \e |A||B|$, and such that there is no $E_0/E_1$-copy of $U(k)$ in $G$. 

Then there is are partitions $\calP_A$ of $A$ and $\calP_B$ of $B$ so that $\calP=(\calP_A,\calP_B)$ is $2\delta^{1/16}$-homogeneous for $G$ with respect to $\{E_0,E_1\}$ and so that $\calP$ has size at most $2c(\delta/8)^{-k}$.
\end{proposition}

Proposition \ref{prop:vcremoval} can be deduced from results of Alon, Fischer, and Newman from  \cite{Alon.2007} (with slightly different bounds).  In particular, Proposition \ref{prop:vcremoval} follows from an asymmetric version of Lemma 1.6 of \cite{Alon.2007}, which the authors state also holds in an asymmetric setting via an identical proof.

For the sake of completeness, and because some amount of translation is required to move from \cite{Alon.2007} to Proposition \ref{prop:vcremoval}, we also include a proof in this paper. The proof uses an analogue of the packing lemma for edge colored graphs, adapted from an application in \cite{Fox.2017bfo}.  The usual Haussler packing lemma says that for any graph $G=(V,E)$ of small VC-dimension, there exist a small number of vertices $x_1,\ldots, x_s$ so that every $v\in V$ looks like one of the $x_i$, in the sense that $N(v)\Delta N(x_i)$ is small (for more on this, see \cite{Fox.2017bfo}).  In the setting of bipartite edge-colored graphs, something similar holds with respect to the ``non-error" edge colors.  A lemma to this effect was proved in \cite{Terry.2022}. 

\begin{lemma}[Lemma 2.14 in \cite{Terry.2022}]\label{lem:hausslercor}
For all $k\geq 1$ there is a constant $c=c(k)$ so that the following holds.  Suppose $\delta,\e>0$ satisfy $\e\leq c^{-2}(\delta/8)^{2k+2}$.    Assume $G=(A\cup B, E_0,E_1,E_2)$ is a bipartite edge-colored graph such that there is no $E_0/E_1$-copy of $U(k)$ in $G$, and such that $|E_2|\leq \e |A||B|$.  

Then there is an integer $m\leq 2c(\delta/8)^{-k}$,  a subset $U\subseteq A$ with $|U|\leq \sqrt{\e} |A|$, and  vertices $x_1,\ldots, x_m\in A\setminus U$ so that for all $a\in A\setminus U$, $|N_{E_2}(a)|\leq \sqrt{\e}|B|$ and for some $1\leq i\leq m$, $\max\{|N_{E_1}(a)\Delta N_{E_1}(x_i)|,|N_{E_0}(a)\Delta N_{E_0}(x_i)|\}\leq \delta |B|$. 
\end{lemma}

Using a proof similar to \cite{Fox.2017bfo}, we now use Lemma \ref{lem:hausslercor} to prove Proposition \ref{prop:vcremoval}.  

\vspace{2mm}

\noindent{\bf Proof of Proposition \ref{prop:vcremoval}.} By Lemma \ref{lem:hausslercor}, there is an integer $m\leq 2c(\delta/8)^{-k}$, a subset $U\subseteq A$ with $|U|\leq \sqrt{\e} |A|$, and vertices $x_1,\ldots, x_m\in A\setminus U$ so that for all $a\in A\setminus U$, $|N_{E_2}(a)|\leq \sqrt{\e}|B|$ and for some $1\leq i\leq m$, $\max\{|N_{E_1}(a)\Delta N_{E_1}(x_i)|,|N_{E_0}(a)\Delta N_{E_0}(x_i)|\}\leq \delta|B|$.  Let 
$$
B_{err}=\bigcup_{i=1}^mN_{E_2}(x_i). 
$$
Note $|B_{err}|\leq m\sqrt{\e}|B|<\delta |B|$, where the second inequality is due to the upper bound on $m$ and our assumption on $\e$.  To ease notation, let $B'=B\setminus B_{err}$ and $A'=A\setminus U$.

We now consider the restriction $G[A', B']=(A'\cup B', E_0[A',B],E_1[A',B'], E_2[A',B'])$.  Note that 
$$
|E_2[A',B']|\leq |E_2\cap K_2[A',B]|\leq \sum_{x\in A'}|N_{E_2}(a)|\leq \sqrt{\e}|B|\leq 2\sqrt{\e}|B'|,
$$
where the last inequality is because $|B'|\geq (1-\delta)|B|$, and $\delta<1/2$.  Clearly $G[A', B']$ also contains no $E_0[A',B']/E_1[A',B']$-copy of $U(k)$.  Therefore, we can apply Lemma \ref{lem:hausslercor} to $G[A', B']$ with parameters $k$, $2\sqrt{\e}$ and $\delta$.  This yields an integer $p\leq 2c(\delta/8)^{-k}$, a subset $V\subseteq B'$ with $|V|\leq 2\e^{1/4} |B'|$, and vertices $y_1,\ldots, y_p\in B'\setminus V$  so that for all $b\in B'\setminus V$, $|N_{E_2}(b)|\leq 2\e^{1/4}|A'|$ and for some $1\leq i\leq p$,
$$
\max\{|(N_{E_1}(b)\Delta N_{E_1}(y_i))\cap A'|,|(N_{E_0}(b)\Delta N_{E_0}(y_i))\cap A'|\}|\leq \delta|A'|.
$$ 
We now define partitions of $A$ and $B$.  First, for each $b\in B'\setminus V$, let 
$$
g(b)=\min\{j\in [p]: \max\{|N_{E_1}(b)\Delta N_{E_1}(y_j)|,|N_{E_0}(b)\Delta N_{E_0}(y_j)|\}\leq \delta |A'|\}.
$$
Then let $B_0=B_{err}\cup V$, and for each  $j\in [p]$, define
$$
B_j=\{b\in B\setminus B_0: g(b)=j\}.
$$
This yields a partition $\calP^B=\{B_0,B_1,\ldots, B_p\}$ of $B$.  Note that by construction, 
$$
|B_0|\leq (\delta +2\e^{1/4})|B|<2\delta|B|,
$$
where the last inequality is by assumption on $\e$.  Define
$$
A_0=U\cup \bigcup_{i=1}^pN_{E_2}(y_i).
$$
Observe $|A_0|\leq \sqrt{\e}|A|+2p\e^{1/4}|A'|\leq 2\delta |A|$, where the inequality is by the bound on $p$, and our assumptions on $\e$ and $\delta$.  Then, for each $a\in A\setminus A_0$, let 
$$
f(a)=\min\{i\in [m]: \max\{|N_{E_1}(a)\Delta N_{E_1}(x_i)|,|N_{E_0}(a)\Delta N_{E_0}(x_i)|\}\leq \delta |B|\}.
$$
Now for each $i\in [m]$,  let 
$$
A_i=\{a\in A\setminus A_0: f(a)=i\}.
$$
We now have a partition $\calP^A=\{A_0,A_1,\ldots, A_m\}$ of $A$.  We show $\calP=(\calP_A,\calP_B)$ has the desired properties.  To ease notation in the verification, we let $A''=A\setminus A_0$ and $B''=B\setminus B_0$.  For each $1\leq i\leq m$ and $x\in A_{i}$, let 
$$
\Gamma_A(x)=\Big(N_{E_2}(x)\cup (N_{E_0}(x)\Delta N_{E_0}(x_{i}))\cup (N_{E_1}(x)\Delta N_{E_1}(x_{i}))\Big)\cap B''.
$$
By construction, for all $x\in A''$, 
$$
|\Gamma_A(x)|\leq \sqrt{\e}|B|+2\delta |B|\leq 8\delta |B''|,
$$
where the second inequality is by the lower bound on $B''$ and our assumptions on $\e,\delta$.  Now let $\Gamma_A=\bigcup_{x\in A''}\Gamma_A(x)$. By above, $|\Gamma_A|\leq 10\delta |A''||B''|$.   Similarly, for each $1\leq j\leq p$ and $y\in B_{j}$, let 
$$
\Gamma_B(y)=\Big((N_{E_0}(y)\Delta N_{E_0}(b_{j}))\cup (N_{E_1}(y)\Delta N_{E_1}(b_{j}))\Big)\cap A''.
$$
By construction, we have that for all $b\in B''$, $|\Gamma_B(x)|\leq 2\delta |A''|$. Now let $\Gamma_B=\bigcup_{b\in B''}\Gamma_B(b)$, and observe that $|\Gamma_B|\leq 2\delta |A''||B''|$.  Setting, $\Gamma=\Gamma_A\cup \Gamma_B$ the above implies $|\Gamma|\leq 10\delta |A''||B''|$.  We now define three sets of pairs from $\calP_A\times \calP_B$.
\begin{align*}
\Sigma_1&=\{(A_{i},B_{j})\in \calP_A\times \calP_B:i=0\text{ or }j=0\},\\
\Sigma_2&=\{(A_{i},B_{j})\in \calP_A\times \calP_B:|\Gamma\cap K_2[A_{i},B_{j}]|\leq 2\sqrt{\delta}|A_{i}||B_{j}|\}\text{ and }\\
\Sigma_3&=\{(A_{i},B_{j})\in \calP_A\times \calP_B:|A_i|<2\sqrt{\delta}|A|/m\text{ or }|B_j|<2\sqrt{\delta}|B|/p\}
\end{align*}
We observe that the upper bounds on $|A_0|$ and $|B_0|$ imply
$$
|\bigcup_{(X,Y)\in \Sigma_1}K_2[X,Y]|\leq |A_0||B|+|A||B_0|\leq 4\delta |A||B|.
$$
Since $|\Gamma|\leq 10\delta |A''||B''|$, Lemma \ref{lem:averaging} implies
$$
|\bigcup_{(X,Y)\in \Sigma_2}K_2[X,Y]|\leq 10\sqrt{\delta}|A||B|.
$$
Finally, it is straightforward to see that by definition of $\Sigma_3$, 
$$
|\bigcup_{(X,Y)\in \Sigma_3}K_2[X,Y]|\leq 4\sqrt{\delta} |A||B|.
$$
Setting $\Sigma=(\calP_A\times \calP_B)\setminus (\Sigma_1\cup \Sigma_2\cup \Sigma_3)$, we have by the inequalities above that
$$
|\bigcup_{(X,Y)\in \Sigma}K_2[X,Y]|\geq (1-18\sqrt{\delta})|A||B|>(1-\delta^{1/16})|A||B|.
$$
Thus it suffices to show that for every $(X,Y)\in \Sigma$, there is some $\alpha\in \{0,1\}$ so that $|E_{\alpha}\cap K_2[X,Y]|\geq (1-\delta^{1/16})|X||Y|$.
 
To this end, fix some $(X,Y)\in \Sigma$.  By construction, there are some $1\leq i\leq m$ and $1\leq j\leq p$ so that $X=A_{i}$ and $Y=B_{j}$.  Define
$$
Y_0=N_{E_0}(x_i)\cap B_j \text{ and }Y_1=N_{E_1}(x_i)\cap B_j.
$$
Since $(X,Y)\in \Sigma$, we have that for each $\alpha\in \{0,1\}$,  
\begin{align*}
|E_{\alpha}\cap K_2[A_i,Y_{\alpha}]|\geq |A_i||Y_{\alpha}|-|K_2[A_i,Y_{\alpha}]\setminus \Gamma_A|\geq  |A_i||Y_{\alpha}|-10\delta^{1/2}|A_i||B_j|.
\end{align*}
Consequently, if we know there is some $\alpha\in \{0,1\}$ so that $|Y_{\alpha}|\geq (1-\delta^{1/8})|B_j|$, then we are done, as the above would then imply 
$$
|E_{\alpha}\cap K_2[A_i,Y_{\alpha}]|\geq (1-10\delta^{1/2}-\delta^{1/8})|A_i||B_j|\geq (1-\delta^{1/16})|A_i||B_j|.
$$
Thus, it suffices to show $\max\{|Y_1|,|Y_0|\}\geq (1-\delta^{1/8})|B_j|$.  Assume towards a contradiction $\max\{|Y_1|,|Y_0|\}<(1-\delta^{1/8})|B_j|$, i.e. $\min\{|Y_1|,|Y_0|\}>\delta^{1/8}|B_j|$. Since $|\Gamma\cap K_2[A_i,B_j]|\geq (1-10\delta^{1/2})|A_i||B_j|$, there is a set $Y'\subseteq B_j$ satisfying 
$$
|Y'|\geq (1-10\delta^{1/4})|B_j|,
$$
such that for all $y\in Y'$, $|N_{\Gamma}(y)\cap A_i|\geq (1-\delta^{1/4})|A_i|$.  Combining with the lower bounds on the sizes of $Y_1$ and $Y_0$, the pigeon hole property tells us there exist some $b_0\in Y'\cap Y_0$ and some $b_1\in Y'\cap Y_1$.  We now have that $x_ib_0\in E_0$ while $x_ib_1\in E_1$. However, we also have $x_ib_0,x_ib_1\in \Gamma_B$, which implies 
$$
x_ib_0\in E_0\Leftrightarrow x_iy_j\in E_0\Leftrightarrow x_ib_1\in E_0.
$$
We have now arrived at a contradiction.
\qed

\section{Proof of the Main Theorem} 

In this section we will prove Theorem \ref{thm:slvccor}.  It will follow from the following slightly less general statement.

\begin{theorem}\label{thm:slvc}
For all $k\geq 1$ there are $K,\e^*>0$ so that for all $0<\e<\e^*$ the following holds. Suppose $H=(U\cup V\cup W,F)$ is a sufficiently large tripartite $3$-graph with $|U|=|V|=|W|$ and with slicewise VC-dimension less than $k$. Then there exists an $\tau$-homogeneous partition $\calP$ for $H$ of size at most $2^{2^{\tau^{-K}}}$.  

Moreover, $\calP$ has the form $\calP_U\cup \calP_V\cup \calP_W$ where $\calP_U,\calP_V,\calP_W$ are almost $\tau^{64}/28$-good partitions of $U,V,W$, respectively.
\end{theorem}

We can easily deduce Theorem \ref{thm:slvccor} from Theorem \ref{thm:slvc}.

\vspace{2mm}

\noindent{\bf Proof of Theorem \ref{thm:slvccor} from Theorem \ref{thm:slvc}.}
Let $K,\e^*$ be as in Theorem \ref{thm:slvc} and $0<\e<(\e^*)^2$.  Suppose $H=(V,E)$ is a sufficiently large $3$-graph with slicewise VC-dimension less than $k$.  Let $H'$ be the $3$-partite $3$-graph  $(A\cup B\cup C,E(H'))$, where $A=\{a_v: v\in V\}$, $B=\{b_v: v\in V\}$ and $C=\{c_v: v\in V\}$, and 
$$
E(H')=\{a_vb_{v'}c_{v''}: vv'v''\in E\}.
$$
Note $H'$ has slicewise VC-dimension less than $k$, and $|A|=|B|=|C|$.  By Theorem \ref{thm:slvc}, there is an $\e^2$-homogeneous partition $\calP=\calP_A\cup \calP_B\cup \calP_C$ of $H'$ of size $t\leq 2^{2^{\e^{-2K}}}$.  Say $\calP_A=\{A_1,\ldots, A_{t_1}\}$, $\calP_B=\{B_1,\ldots, B_{t_2}\}$ and $\calP_C=\{C_1,\ldots, C_{t_3}\}$.

For each $A_i,B_j,C_s\in \calP$, let $A_i'=\{v\in V: a_v\in A_i\}$, $B_j'=\{v\in V: b_v\in B_j\}$ and $C_s'=\{v\in V: c_v\in C_s\}$.  Let $\calQ$ be the partition of $V$ which is the common refinement of the partitions $\{A_1',\ldots, A_{t_1}'\}$, $\{B_1',\ldots, B_{t_2}'\}$, and $\{C_1',\ldots, C_{t_3}'\}$. Clearly
$$
|\calQ|\leq t_1t_2t_3\leq (2^{2^{\e^{-2K}}})^3=2^{3\cdot 2^{\e^{-2K}}}.
$$
Clearly there is some $K'>0$ depending only on $K$ so that for all sufficiently small $\e>0$, $2^{3\cdot 2^{\e^{-2K}}}<2^{2^{\e^{-K'}}}$.  Thus, it suffices to show $\calQ$ is $\e$-homogeneous with respect to $H$.

For each $X\in \calQ$, let $X^A=\{a_v: v\in X\}$, $X^B=\{b_v: v\in X\}$, $X^C=\{c_v: v\in X\}$.  Then $\calP'=\{X^A,X^B,X^C: X\in \calQ\}$ is a partition of $V(H')$ refining $\calP$.  By Lemma \ref{lem:averaging}, $\calP'$ is $\e$-homogeneous with respect to $H'$.  Now let 
$$
\Sigma=\{(X,Y,Z)\in \calQ^3: \text{$(X^A,Y^B,Z^C)$ is $\e$-homogeneous with respect to $H'$}\}
$$
For all $(X,Y,Z)\in \Sigma$, $d_H(X,Y,Z)=d_{H'}(X^A,Y^B,Z^C)\in [0,\e)\cup (1-\e,1]$.  Further, 
$$
|\bigcup_{(X,Y,Z)\in \Sigma}X\times Y\times Z|\geq |\bigcup_{(X,Y,Z)\in \Sigma}X^A\times Y^B\times Z^C|\geq (1-\e)|A||B||C|=(1-\e)|V(H)|^3.
$$
This shows $\calQ$ is an $\e$-homogeneous partition of $H$, as desired.
\qed

\vspace{2mm}

The rest of the paper will focus on proving Theorem \ref{thm:slvc}.  The proof of Theorem \ref{thm:slvc} is rather long, so we now give an outline of the overall strategy.  The starting point is a $3$-partite $3$-graph  $H=(U\cup V\cup W, E)$ with $|U|=|V|=|W|$ and slicewise VC-dimension less than $k$, and the goal is to construct a homogeneous partition for $H$ with few parts.

\vspace{2mm}

\noindent\underline{Step 1 (Reduced Graphs for Slice Graphs)}:  Since $H$ has slicewise VC-dimension less than $k$, we have that for all $x\in U$, $H_x$ has VC-dimension less than $k$ (see Definition \ref{def:slicegraph}).  By Theorem \ref{thm:bipfoxpachsuk}, there is an homogeneous equipartition $\calP_x$ of $H_x$ refining $\{V,W\}$ with a small number of parts.  We then construct a bipartite edge-colored graph $\calG_x$ as discussed in Section \ref{ss:ecg}. In particular, the  vertices of $\calG_x$ are the parts of $\calP_x$, and the edge colors $E_0^x,E_1^x,E_2^x$ correspond to the dense, sparse, and non-homogeneous pairs of $\calP_x$, respectively.

\vspace{2mm}

\noindent\underline{Step 2 (Group vertices based on Reduced Graphs)}: Partition $U$ into equivalence classes, where  $x\equiv x'$ if and only if $\calG_x=\calG_{x'}$.

\vspace{2mm}

\noindent\underline{Step 3 (Structure Theorem on Equivalence Classes)}: For each $\equiv$-equivalence class $X$ in $U$, find a partition $\calP_X^V$ of $V$ which is almost $\e$-good with respect to $H[X, V, W]$.

\vspace{2mm}

\noindent\underline{Step 4 (Refine)}:  Let $\calP^V$ be the partition of $V$ obtained by taking the common refinement of the partitions $\calP_X^V$ from step 3, as $X$ ranges over the $\equiv$-equivalence classes in $U$.  Show this yields an almost $\e$-good partition $\calP_V$ of $V$ with respect to $H$.

\vspace{2mm}

\noindent\underline{Step 5 (Repeat)}: Repeat Steps 1-4 with the roles of $U,V,W$ permuted to obtain almost $\e$-good partitions $\calP_W$ of $W$ and $\calP_U$ of $U$.

\vspace{2mm}

\noindent\underline{Step 7 (Symmetry)}: Use Proposition \ref{prop:goodhom} to conclude $\calP=\calP_U\cup \calP_V\cup \calP_W$ is a homogeneous partition with respect to $H$.

\vspace{2mm}

As the reader can see, there are many steps to this argument.  The hardest is Step 3.  We choose to encapsulate this step into its own result, Proposition \ref{prop:mainslice}, which we will state and prove in the next subsection, Subsection \ref{ss:step3}.  In the final subsection, Section \ref{ss:final}, we will prove Theorem \ref{thm:slvc} using Proposition \ref{prop:mainslice}.

\subsection{Accomplishing Step 3}\label{ss:step3}

In this section, we prove the main lemma, Proposition \ref{prop:mainslice}, which is required to execute Step 3 of the strategy outlined above.  The goal of Proposition \ref{prop:mainslice} is to give a structure theorem for any  tripartite $3$-graph $H=(U\cup V\cup W,F)$ with slicewise VC-dimension less than $k$, under the assumption that there exists a single bipartite edge-colored graph $\calG$ so that for all $x\in U$, $\calG$ arises as the reduced graph associated to a homogeneous partition of $H_x$.

\begin{proposition}\label{prop:mainslice}
Suppose $1\leq s,r\leq \ell$, $k\geq 1$, and $c=c(k)$ is as in Proposition \ref{prop:vcremoval}.   Let $\e>0$ be sufficiently small compared to $c$, $1/\ell$ and $1/k$, and set $\mu=8(16\e^{1/16}c^4)^{1/4k+4}$.  Assume the following hold.
\begin{enumerate}
\item\label{sl:0} $H=(U\cup V\cup W,F)$ is a sufficiently large tripartite $3$-graph with $\SVC(H)<k$,
\item $\calG=(P\cup Q, F_0,F_1,F_2)$ is a bipartite edge-colored graph with $P=\{p_0,\ldots, p_r\}$ and $Q=\{q_0,\ldots, q_s\}$, 
 \item\label{sl:1} For all $x\in U$, there is a partition $V=B_x^0\cup \ldots \cup B_x^r$ and $W=C_x^0\cup \ldots \cup C_x^s$ so that the following hold, where $E(H)^1:=E(H)$ and $E(H)^0:={V(H)\choose 3}\setminus E(H)$.
 \begin{enumerate}
 \item\label{sl:2} For each $\alpha\in \{0,1\}$ and $p_iq_j\in F_{\alpha}$, $|E(H_x)^{\alpha}\cap K_2[B_x^i,C_x^j]|\geq (1-\e)|B_x^i||C_x^j|$,
 \item\label{sl:3} For each $1\leq i\leq r$ and $1\leq j\leq s$, $|B_x^i|=|C_x^j|=|W|/\ell$,
 \item\label{sl:4} $|B_x^0|\leq \e |V|$ and $|C_x^0|\leq \e |W|$.
 \end{enumerate}
  \item\label{sl:5} For all $1\leq i\leq r$ and $1\leq j\leq s$, $|N_{F_2}(p_i)|\leq \e |Q|$ and $|N_{F_2}(q_j)|\leq \e|P|$.
\end{enumerate}
Then there exists a partition $\calP_V$ of $V$ which is almost $6\mu^{1/512}$-good with respect to $H$ satisfying
$$
|\calP_V|\leq  (2c(\mu/8)^{-k})^{\ell(2\ell)^{\ell}\e^{-\ell}}.
$$
\end{proposition}

Note that hypotheses (\ref{sl:1}) and (\ref{sl:5}) in Proposition \ref{prop:mainslice} tell us that for each $x\in U$, the partition $\{B_x^0\cup \ldots \cup B_x^r,C_x^0\cup \ldots \cup C_x^s\}$ is an $\e$-homogeneous partition for the slice graph $H_x$. Before proving Proposition \ref{prop:mainslice}, we require two lemmas.  The first, Lemma \ref{lem:cover} below, tells us that if $G=(A\cup B,E)$ is a bipartite graph, then we can cover most of $A$ with sets which are either contained the neighborhood of a single element in $G$, or incident to few edges in $G$.  

\begin{lemma}\label{lem:cover}
Suppose $G=(A\cup B, E)$ is a bipartite graph and $\delta>0$.  There exists an integer $t\leq \delta^{-1/2}$ and a partition $A=A_{err}\cup A_0\cup A_1\cup \ldots\cup A_t$ of $A$ so that the following hold.
\begin{enumerate}[(i)]
\item $|A_{err}|\leq \delta^{1/4}|A|$,
\item For each $1\leq i\leq t$, there is a vertex $b_i\in B$ so that $A_i\subseteq N_G(b_i)$ and $|A_i|=\sqrt{\delta}|A|$,
\item For all $x\in A_0$, $|N_G(x)|\leq \sqrt{\delta}|B|$,
\end{enumerate}
\end{lemma}

The proof of Lemma \ref{lem:cover} a simple greedy algorithm, and appears in the appendix.  The second lemma, Lemma \ref{lem:refinement}, is about refinements of almost $\e$-good partitions.  It is essentially a corollary of Lemma \ref{lem:averaging}.

\begin{lemma}\label{lem:refinement}
Suppose $H=(V_1\cup V_2\cup V_3,E)$ is a $3$-partite $3$-graph.  Suppose $\{i_1,i_2,i_3\}=\{1,2,3\}$.  Suppose $\calP$ is a almost $\e$-good partition of $V_{i_1}$ and $\calQ\preceq \calP$.  Then $\calQ$ is an almost $\sqrt{\e}$-good partition of $V_{i_1}$. 
\end{lemma}
\begin{proof}
Let $\calP_{good}$ be the set of almost $\e$-good sets in $\calP$.  Since, by assumption, $\calP$ is $\e$-good with respect to $H$, we have $|\bigcup_{X\in \calP_{good}}X|\geq (1-\e)|V_{i_1}|$.  For each $X\in \calP_{good}$, let 
\begin{align*}
\Gamma_1(X)&=\{yz\in K_3[V_{i_2},V_{i_3}]: |N_H(yz)\cap X|\geq (1-\e)|X|\}\text{ and }\\
\Gamma_0(X)&=\{yz\in K_3[V_{i_2},V_{i_3}]: |N_H(yz)\cap X|\leq \e |X|\}.
\end{align*}
and then set
$$
\Omega(X)=(\bigcup_{yz\in \Gamma_1(X)}K_3[\{y\},\{z\},N_H(yz)\cap X]\Big)\cup (\bigcup_{yz\in \Gamma_0(X)}K_3[\{y\},\{z\},X\setminus N_H(yz)]\Big).
$$
Note that for all $X\in \calP_{good}$, $|\Omega(X)|\geq (1-\e)|X||V_{i_2}||V_{i_3}|$.  Consequently, if we let 
$$
\Omega=\bigcup_{X\in \calP_{good}}\Omega(X),
$$
then $|\Omega|\geq (1-\e)|V_{i_1}||V_{i_2}||V_{i_3}|$.  By Lemma \ref{lem:averaging}, there is a subset $\calS\subseteq \calQ$ with the property that $|\bigcup_{S\in \calS}S|\geq (1-\e^{1/2})|V_{i_1}|$ and so that for all $S\in \calS$, 
$$
|\Omega\cap K_3[S,V_{i_2},V_{i_3}]|\geq (1-\e^{1/2})|S||V_{i_2}||V_{i_3}|.
$$
We show every $S\in \calS$ is almost $\e^{1/4}$-good with respect to $H$, which will finish the proof.  To this end, fix $S\in \calS$.  By Lemma \ref{lem:averaging}, if we define 
$$
\calG(S):=\{(y,z)\in V_{i_2}\times V_{i_3}: |N_{\Omega}(yz)\cap S|\geq (1-\e^{1/4})|S|\},
$$
then $|\calG(S)|\geq (1-\e^{1/4})|S|$.  By definition of $\Omega$, and since $S\subseteq X$ for some $X\in \calP_{good}$, we know that for all $(y,z)\in \calG(S)$, $N_{\Omega}(yz)\cap S\subseteq N_H(yz)\cap S$ or $N_{\Omega}(yz)\cap S\subseteq S\setminus N_H(yz)$.  Thus, for every $(y,z)\in \calG(S)$, we either have $|N_H(yz)\cap S|\geq (1-\e^{1/4})|S|$ or $|S\setminus N_H(yz)|\geq (1-\e)|S|$.  This finishes the proof.
\end{proof}

 We now prove Proposition  \ref{prop:mainslice}.

\vspace{2mm}

\noindent{\bf Proof of Proposition  \ref{prop:mainslice}.} Fix $1\leq s,r\leq \ell$, $k\geq 1$, and let $c=c(k)$ be as in Proposition \ref{prop:vcremoval}.   Fix $\e>0$ sufficiently small, and assume $H=(U\cup V\cup W,F)$ is a sufficiently large tripartite $3$-graph and $\calG=(P\cup Q, F_0,F_1,F_2)$ is an edge-colored bipartite graph so that $H$ and $\calG$ satisfy the hypotheses (\ref{sl:0})-(\ref{sl:5}) of Proposition  \ref{prop:mainslice}.  Throughout, the edge-colored bipartite graph serves a similar role as a ``reduced graph" or ``quotient graph" would in traditional applications of the regularity lemma.  To emphasize this, we will refer to it as the auxiliary structure, $\calG$. 

 We start by defining partitions of $K_2[U,V]$ and $K_2[U,W]$ using the sets appearing in assumption (\ref{sl:1}).  In particular, for each $0\leq u\leq r$, $0\leq v\leq s$, let
\begin{align*}
P_u=\{xy\in K_2[U,V]: y\in B^u_x\}\text{ and }Q_v=\{xy\in K_2[U,W]: y\in C^v_x\}.
\end{align*}
Note that $K_2[U,V]=P_0\cup \ldots \cup P_r$ and $K_2[U,W]=Q_0\cup \ldots \cup Q_s$.  By hypothesis (\ref{sl:4}), we have that for all $x\in U$,
$$
|N_{P_0}(x)|\leq \e|V|\text{ and }|N_{Q_0}(x)|\leq \e|W|,
$$
and consequently, 
\begin{align}\label{a}
|P_0|\leq \e|U||V|\text{ and }|Q_0|\leq \e|U||W|.
\end{align}
We will refer to the sets of the form $P_u$, $Q_v$ as edge colors, and we will think of $P_0$ and $Q_0$ as the ``error colors."   Using these edge colors, we now define a tripartite $(U,V,W)$-decomposition $\calD=(\calD_1,\calD_2)$ of $V(H)$.  Let $\calD_1=(\calD_U,\calD_V,\calD_W)$ where $\calD_U=\{U\}$, $\calD_W=\{W\}$, $\calD_V=\{V\}$.  Then let $\calD_2=\{\calD_{UV},\calD_{VW},\calD_{UW}\}$ where 
$$
\calD_{VW}=\{K_2[V,W]\},\text{ }\calD_{UV}=\{P_u:0\leq u\leq  r\},\text{ and }\calD_{UW}=\{Q_v:0\leq v\leq s\}.
$$
  A triad from $\calD$ is then a $3$-partite graph of the following form, where $0\leq u\leq r$ and $0\leq v\leq s$.
$$
T_{uv}=(U\cup V\cup W, P_u\cup Q_v\cup K_2[V,W]).
$$
We now show that certain triads from $\calD$ are homogeneous with respect to $H$ in the sense of Definition \ref{def:triphom}.

Given $1\leq u\leq r$ and $1\leq v\leq s$, let $t(u,v)\in \{0,1,2\}$ be such that, in the auxiliary structure $\calG$,  $p_uq_v\in F_{t(u,v)}$.  When $t(u,v)\in \{0,1\}$, we have by hypothesis (\ref{sl:3}) that 
$$
|K_3^{(2)}(T_{uv})|=\sum_{x\in U}|N_{P_u}(x)||N_{Q_v}(x)|=|U||V||W|/\ell^2,
$$
and by hypothesis  (\ref{sl:2}),
$$
|d_H(T_{uv})-t(u,v)|=\Big|\frac{|E(H)\cap K_3^{(2)}(T_{uv})|}{|K_3^{(2)}(T_{uv})|}-t(u,v)\Big|\leq \e.
$$ 
Using the notation $E(H)^1=E(H)$ and $E(H)^0=K_3[U,V,W]\setminus E(H)$, we can restate the above inequality as follows.
\begin{align}\label{gammagood}
|E(H)^{t(u,v)}\cap K_3^{(2)}(T_{uv})|\geq (1-\e)|K_3^{(2)}(T_{uv})|.
\end{align}
We now define a crucial object for the rest of the proof.  Let  $\Gamma$ be the $3$-partite $3$-graph with vertex set $U\cup V\cup W$ and edge set
$$
E(\Gamma):=\bigcup_{p_uq_v\in F_0\cup F_1}\Big(K^{(2)}_3(T_{uv})\cap E(H)^{t(u,v)}\Big).
$$
Roughly speaking, $E(\Gamma)$ is the set of triples from $K_3[U,V,W]$ which are in a homogeneous triad from $\calD$, and which are moreover ``correct" for their triad.  Note that for any $xyz\in E(\Gamma)$, we can decide whether or not $xyz$ is an edge in $H$ by looking at which triad $T_{uv}$ contains $xyz$ and then whether we have $p_uq_v\in F_0$ or $p_uq_v\in F_1$ in $\calG$.

Our next goal is to prove a consequence of the assumption that $\SVC(H)<k$.  To state this, we must define some auxiliary edge colors, which are obtained by grouping together the existing edge colors between $U$ And $V$, based on properties of $\calG$.   Given $1\leq v\leq s$, define 
\begin{align}\label{pdef}
P_{v}^1=\bigcup_{p_u\in N_{F_1}(q_v)}P_u,\text{       }P_{v}^0=\bigcup_{p_u\in N_{F_0}(q_v)}P_u,\text{ and }P_{v}^2=K_2[U,V]\setminus (P_{v}^1\cup P_{v}^0),
\end{align}
where we note that the sets $N_{F_1}(q_v)$ and $N_{F_0}(q_v)$ are computed in auxiliary structure $\calG$.  This yields a new bipartite edge-colored graph for each $v\in [r]$, namely
$$
G_{v}:=(U\cup V, P_{v}^0\cup P_{v}^1\cup P_{v}^2).
$$
We now show there are restrictions on copies of $U(k)$ appearing in such a $G_v$ (see Definition \ref{def:copies}). 
  
\begin{observation}\label{ob}
Suppose $1\leq v\leq s$, and $\{x_1,\ldots, x_k\}\cup \{y_S:S\subseteq [k]\}$ form a $P_{v}^1/P_{v}^0$-copy of $U(k)$ in $G_{v}$.  Then there DOES NOT exist $c\in W$ so that the following hold.
\begin{enumerate}[(a)]
\item $(\{x_1,\ldots, x_k\}\cup \{y_S:S\subseteq [k]\})\cap U\subseteq N_{Q_v}(c)$ and 
\item For all $1\leq i\leq k$ and $S\subseteq [k]$, $x_iy_S\in N_{\Gamma}(c)$.
\end{enumerate}
\end{observation}
\begin{proof}
Suppose towards a contradiction there exists some $c\in W$ so that (a) and (b) hold.  Without loss of generality, we may assume the  $P_{v}^1/P_{v}^0$-copy of $U(k)$ has right side in $V$, i.e. $x_1,\ldots, x_k\in N_{Q_v}(c)$ and $\{y_S:S\subseteq [k]\}\subseteq V$ (the other case is similar).  By assumption, we now have that for all $i\in [k]$ and $S\subseteq [k]$, $i\in S$ implies $x_iy_S\in P_{v}^1\cap N_{\Gamma}(c)$ and $i\notin S$ implies $x_iy_S\in P_{v}^0\cap N_{\Gamma}(c)$.  By definition of $\Gamma$, this means that for all $i\in S$,  $cx_iy_S\in E(H)$ and for all $i\notin S$, $cx_iy_S\notin E(H)$.  This contradicts that $H$ has slicewise VC-dimension less than $k$.
\end{proof}

The goal of the rest of the proof to create situations where we can use  Observation \ref{ob}.  We give an informal (and slightly inaccurate) outline of the steps we will take.  The reader may choose to skip the outline, and merely refer back to it for reference.

\begin{enumerate}[(I)]
\item Throw away edges and vertices which do not behave well with respect to  $\Gamma$. This is necessary for part (b) in  Observations \ref{ob}. 
\item Build a partition $\calX$ of $U$ so that for all $A\in \calX$, and any $1\leq v\leq s$, there exists some $c\in W$ so that $A\subseteq N_{Q_v}(c)$.  This will allow us to satisfy part (a) of Observation \ref{ob} when we restrict ourselves to $H[A\cup V\cup W]$ for each $A\in \calX$.
\item For each $A\in \calX$ and most $v\in [s]$, use Observation \ref{ob} to define a bipartite edge-colored graph  $(A\cup V,\mathbf{P}_{v,A}^0\cup \mathbf{P}_{v,A}^1\cup \mathbf{P}_{v,A}^2)$ so that $ \mathbf{P}_{v,A}^2$ is small, so that there is no $\mathbf{P}_{v,A}^0/\mathbf{P}_{v,A}^1$-copy of $U(k)$, and so that the triads of the form 
$$
(A\cup V\cup W, Q_v\cup \mathbf{P}_{v,A}^{0}\cup K_2[V,W])\text{ and }(A\cup V\cup W, Q_v\cup \mathbf{P}_{v,A}^{1}\cup K_2[V,W])
$$
are homogeneous with respect to $H$ in the sense of Definition \ref{def:triphom}.  
\item For each $A\in \calX$ and $v\in [s]$ as in the previous step, apply Proposition \ref{prop:vcremoval} to $(A\cup V, \mathbf{P}_{v,A}^0, \mathbf{P}_{v,A}^1, \mathbf{P}_{v,A}^2)$ to obtain vertex partitions $\calP^A_{v,A}$ of $A$ and $\calP_{v,A}^V$ of $V$ so that $(\calP_{v,A}^A,\calP_{v,A}^V)$ are homogeneous with respect to $\{\mathbf{P}_{v,A}^0, \mathbf{P}_{v,A}^1\}$ in the sense of Definition \ref{def:echom}.
\item For each $A\in \calX$, define a tripartite decompositions $\calD^A$ of $A\cup V\cup W$ as follows.  Let $\calD_W^A=\{W\}$ be the trivial partition, and let $\calD^A_A$ and $\calD_V^A$ be the partitions of $A$ and $V$ respectively, generated by taking the common refinements of the partitions of the form $\calP_{v,A}^A$ and $\calP_{v,A}^V$ from the preceding step, as $v$ ranges over most elements of $[s]$.  For each $X\in \calD_A^A$, this decomposition will include partitions of $K_2[X,W]$ roughly equal to $\{Q_v\cap K_2[X,W]: 1\leq v\leq s\}$. All other elements in $\calD^A_2$ are trivial partitions.  Show each $\calD^A$ is homogenous with respect to $H[A,V,W]$ in the sense of Definition \ref{def:triphom}.
\item Show each partition $\calD_V^A$ of $V$ from the previous step is almost good with respect to $H[A,V,W]$ in the sense of Definition \ref{def:almostgood}.  
\item Take the common refinement of the partitions $\calD_V^A$ of $V$,  as $A$ ranges over the elements in $\calX$.  Conclude the result is a partition of $V$ which is almost good with respect to $H$ using Lemma \ref{lem:refinement}.
\end{enumerate}

\vspace{2mm}

\noindent{\bf Step I.} In this step, we restrict our vertex set and adjust the edge colors to ensure good behavior with respect to $\Gamma$.  We first observe that by inequalities (\ref{a}) and (\ref{gammagood}), we have that  
\begin{align}\label{lbgamma}
|E(\Gamma)\cap K_3[U,V,W]|&\geq \sum_{p_uq_v\in F_0\cup F_1}|E(H)^{t(u,v)}\cap K_3^{(2)}(T_{uv})|\nonumber\\
&\geq (1-\e)\sum_{p_uq_v\in F_0\cup F_1}|K_3^{(2)}(T_{uv})|\nonumber\\
&\geq (1-\e)(|U||V||W|-|P_0||W|-|Q_0||V|)\nonumber\\
&\geq (1-3\e)|U||V|||W|.
\end{align}
Crucially, this means $\Gamma$ covers almost all triples from $K_3[U,V,W]$.  We will restrict our attention to vertices which are sufficiently well behaved with respect to $\Gamma$.  First, we set 
\begin{align*}
W_{good}&=\{c\in W: |N_{\Gamma}(c)|\geq (1-\e^{1/2})|U||V|\},\text{ and }\\
V_{good}&=\{b\in V:|N_{\Gamma}(b)|\geq (1-\e^{1/2})|U||W|\}.
\end{align*}
Inequality (\ref{lbgamma}) implies 
\begin{align}\label{al:vgoodwgood}
|V_{good}|\geq (1-3\e^{1/2})|V|\text{ and }|W_{good}|\geq (1-3\e^{1/2})|W|.
\end{align}
  We now define $G^{good}$ to be the tripartite graph with vertex set $U\cup V_{good}\cup W_{good}$, and edge set
\begin{align*}
E(G^{good})= &\{xy\in K_2[U,V_{good}]: |N_{\Gamma}(xy)|\geq (1-\e^{1/4})|W_{good}|\} \\
&\cup \{xz\in K_2[U,W_{good}]: |N_{\Gamma}(xz)|\geq (1-\e^{1/4})|V_{good}|\}.
\end{align*}
Using inequalities (\ref{lbgamma}) and (\ref{al:vgoodwgood}), standard arguments show
\begin{align}
|E(G^{good})\cap K_2[U,V_{good}]|&\geq (1-3\e^{1/4})|U||V|\text{ and }\label{al:good}\\
|E(G^{good})\cap K_2[U,W_{good}]|&\geq (1-3\e^{1/4})|U||W|.\label{al:good2}
\end{align}
We now collect the remaining the pairs outside $E(G^{good})$ and add them to $P_0,Q_0$ to make slightly larger ``error colors." In particular, let 
\begin{align*}
Q^{bad}=Q_0\cup \Big(K_2[U,W]\setminus E(G^{good})\Big)\text{ and }P^{bad}=P_0\cup \Big(K_2[U,V]\setminus E(G^{good})\Big).
\end{align*}
By inequalities (\ref{a}), (\ref{al:good}), and (\ref{al:good2}),
\begin{align}\label{error}
|Q^{bad}|\leq 4\e^{1/4}|U||W|\text{ and }|P^{bad}|\leq 4\e^{1/4}|U||V|.
\end{align}
  We now restrict ourselves to vertices in $U$ which mostly avoid $Q^{bad}$ and $P^{bad}$.  Let
$$
U_{good}=\{x\in U: |N_{Q^{bad}}(x)|\leq \e^{1/8}|W_{good}|\text{ and }|N_{P^{bad}}(x)|\leq \e^{1/8}|V_{good}|\}.
$$
By inequality (\ref{error}) and standard arguments, $|U_{good}|\geq (1-4\e^{1/8})|U|$.  Finally, we will restrict the edge colors between $U$ and $W$ to the ``good" pairs.  Specifically, for each $1\leq v\leq s$, let 
\begin{align*}
Q_v^{good}&=Q_v\cap K_2[U,W]\cap E(G^{good}).
\end{align*}
Note that by construction, $K_2[U,W]=Q^{bad}\cup Q_1^{good}\cup \ldots \cup Q_s^{good}$.  This finishes Step I.

\vspace{2mm}

\noindent{\bf Step II.} Our next goal is to partition most of $U_{good}$ into pieces that are covered by neighborhoods of the form $Q^{good}_v(c)$ for some $c\in W_{good}$.   We will do this for each edge color $Q_v^{good}$ individually and then intersect the resulting partitions.  

Let $\delta=\e^{1/2}/2\ell$. For each $1\leq v\leq s$, apply  Lemma \ref{lem:cover}, to $(U_{good}\cup W_{good},Q_v^{good})$ to obtain an integer $0\leq K_{v}\leq \delta^{-1/2}$ and a partition $\calX_{Q_v}=\{X_v(err), X_{v}(0),\ldots, X_{v}(K_{v})\}$ of $U_{good}$ so that the following hold.
\begin{enumerate}[(i)]
\item $|X_v(err)|\leq \delta^{1/4}|U_{good}|$,
\item For each $1\leq j\leq K_v$, there is some $c_{v,j}\in W_{good}$ so that $X_v(j)\subseteq N_{Q^{good}_v}(c_{v,j})$ and $|X_v(j)|=\delta^{1/2}|U_{good}|$,
\item For all $x\in X_v(0)$, $|N_{Q^{good}_v}(x)\cap W_{good}|\leq \delta^{1/4}|W_{good}|$.
\end{enumerate}

Let  $\calX$ be the partition of $U_{good}$ obtained by taking the common refinement the partitions of the form $\calX_{Q_v}$ for $v\in [s]$.  Note
\begin{align}\label{pui}
|\calX|\leq (\delta^{-1/2}+2)^{s}\leq (\delta^{-1})^s\leq \delta^{-\ell}\leq (2\ell)^{\ell}\e^{-\ell}.
\end{align}
We now collect certain elements from $\calX$ which we need to throw away. First, set
\begin{align*}
\calX^{err}=\{A\in \calX: |A|\leq \e |U_{good}|/|\calX|\}&\cup \{A\in \calX: \text{ for some }v\in [s], A\subseteq X_v(err)\}.
\end{align*}
Now let 
\begin{align}\label{al:udef}
U_{err}=\Big(U\setminus U_{good}\Big)\cup \Big(\bigcup_{A\in \calX^{err}}A\Big)\text{ and }U_{good}'=U\setminus U_{err}.
\end{align}
Since $|U_{good}|\geq (1-\e^{1/8})|U|$, and by definition of $\calX_{err}$, we have the following upper bound on $|U_{err}|$.
\begin{align*}
|U_{err}|\leq \e^{1/8}|U|+\e |U|+\sum_{v\in [s]}|X_v(err)|\leq 2\e^{1/8}|U|+\ell \delta^{1/4}|U|\leq 3\e^{1/8} |U|,
\end{align*}
where the last inequality is by definition of $\delta$.  This finishes Step II.

\vspace{2mm}

\noindent{\bf Step III.} We now prove a structure theorem for $H[A,V,W]$ for each $A\in \calX\setminus \calX_{err}$.  We do this in Claim \ref{cl:a} below.  Roughly speaking, Claim \ref{cl:a} says we can group together edge colors to make several homogeneous decompositions for $H[A,V,W]$, one for most elements of $[s]$.

\begin{claim}\label{cl:a}
For all $A\in \calX\setminus \calX_{err}$, there exist the following.
\begin{itemize}
\item A set $\calS^A\subseteq [s]$,
\item A partition $K_2[A,W]=\mathbf{Q}^A_0\cup \bigcup_{v\in \calS^A}\mathbf{Q}^A_v$,
\item For each $v\in \calS^A$, a partition 
$$
K_2[A,V]=\mathbf{P}_{A,v}^0\cup \mathbf{P}_{A,v}^1\cup \mathbf{P}_{A,v}^2,
$$
\end{itemize}
so that the following hold.
\begin{enumerate}
\item For all $x\in A$, $|N_{\mathbf{Q}^A_0}(x)|\leq 5\e^{1/8}|W|$, 
\item For all $v\in \calS^A$ and $x\in A$, $|N_{\mathbf{P}_{A,v}^2}(x)|\leq 4\e^{1/16}|V|$,  
\item For all $v\in \calS^A$, there is no $\mathbf{P}_{A,v}^1/\mathbf{P}_{A,v}^0$-copy of $U(k)$ in $(A\cup V, \mathbf{P}_{A,v}^0,\mathbf{P}_{A,v}^1,\mathbf{P}_{A,v}^2)$,
\item Given $v\in \calS^A$ and $\tau\in \{0,1,2\}$, let $T(\mathbf{P}_{A,v}^{\tau},\mathbf{Q}^A_v)$ denote the triad
$$
T(\mathbf{P}_{A,v}^{\tau},\mathbf{Q}^A_v):=(A\cup W\cup V, \mathbf{P}_{A,v}^{\tau}\cup \mathbf{Q}^A_v\cup K_2[V,W]).
$$
Then $T(\mathbf{P}_{A,v}^{\tau},\mathbf{Q}^A_v)$ is $2\e^{1/16}$-homogeneous with respect to $H$ whenever $\tau\neq 2$ and $\mathbf{P}_{A,v}^{\tau}\neq \emptyset$.
\end{enumerate}
\end{claim}
\begin{proof}
Fix $A\in \calX\setminus \calX_{err}$.  We begin by defining  the distinguished sets of edge colors $\calS^A$ for $A$.  In particular, define
\begin{align*}
&\calS^A=\{v\in [s]: A\cap X_{v}(0)=\emptyset\}\text{ and }\calS^A_{err}=\{v\in [s]: A\subseteq X_{v}(0)\}.
\end{align*}
Since $A\notin \calX^{err}$, we have $\calS^A\cup \calS^A_{err}=[s]$.  For each $v\in \calS^A$, define
\[
\mathbf{Q}^A_v= Q_v^{good}\cap K_2[A,W],
\]
and then let $\mathbf{Q}_0^A=K_2[A, W]\setminus(\bigcup_{v\in [s]}\mathbf{Q}_v^A)$. Using that $A\subseteq U_{good}'$ (recall its definition from (\ref{al:udef})), the definition of $\calS^A_{err}$, and property (iii) from the construction of the sets $X_v(0)$ in Step II, we have that for all $x\in A$, 
\begin{align}\label{c}
|N_{\mathbf{Q}^A_0}(x)|&\leq |N_{Q^{bad}}(x)|+|\sum_{v\in \calS^A_{err}}N_{Q^{good}_v}(x)\cap W_{good}|+|W\setminus W_{good}|\nonumber\\
&\leq \e^{1/8} |W|+|\calS^A_{err}|\delta^{1/4}|W|+3\e^{1/2}|W|\nonumber\\
&\leq (4\e^{1/8}+\ell\delta^{1/4})|W|\nonumber\\
&\leq 5\e^{1/8}|W|,
\end{align}
where the last inequality is by definition of $\delta$.  Thus we have shown conclusion (1) of Claim \ref{cl:a}.  Note this implies $|\mathbf{Q}_0^A|\leq 5\e^{1/8}|A||W|$, and therefore, we now know that most of $K_2[A,W]$ is covered by the sets of the form $\mathbf{Q}^A_v$ for $v\in \calS^A$.

Our next goal is to define the edge colors $\mathbf{P}_{v,A}^0,\mathbf{P}_{v,A}^1,\mathbf{P}_{v,A}^2$ for each $v\in \calS^A$.  By definition of $\calS^A$, we have that for each $v\in \calS^A$, $A\subseteq X_v(j_v)\subseteq N_{Q_v^{good}}(c_{v,j_v})$ for $1\leq j_v\leq K_v$.  We let $\calC^A$ denote the set of these vertices $c_{v,j_v}$, i.e. define
\begin{align*}
\calC^A:=\{c_{v,j_v}: v\in \calS^A \}.
\end{align*}
We now define two special subsets of $\calC^A$ based on behavior in the auxiliary structure $\calG$.
\begin{align*}
\calC^A_0=\{c_{v,j_v}\in \calC^A: |N_{F_1}(q_v)|\leq \e^{1/16}r\},\text{ and }\calC^A_1=\{c_{v,j_v}\in \calC^A: |N_{F_0}(q_v)|\leq \e^{1/16} r\}.
\end{align*}
We do not expect $\calC^A_0\cup \calC_1^A$ to necessarily cover all of $\calC^A$.  Recall that by hypothesis (\ref{sl:5}), we know that for every $v\in [s]$, $|N_{F_2}(q_v)|\leq \e r$.  Consequently, $\calC^A_0\cap \calC^A_1=\emptyset$.  We will next define the partition $K_2[A,V]=\mathbf{P}_{v,A}^0\cup \mathbf{P}_{v,A}^1\cup \mathbf{P}_{v,A}^2$ for each $v\in \calS^A$.  How we define this partition depends on whether $c_{v,j_v}$ is in one of $\calC^A_1$, $\calC^A_0$, or neither (cases (a), (b), and (c) below).  This definition will also make use of the edge-colors $P_v^0,P_v^1,P_v^2$ which were defined in displayed equality (\ref{pdef}).
\begin{enumerate}[(a)]
\item If $c_{v,j_v}\in \calC^A_1$, let $\mathbf{P}_{v,A}^0=\emptyset$, and define
\begin{align*}
\mathbf{P}_{v,A}^1&=P_v^1\cap E(G^{good})\cap N_{\Gamma}(c_{v,j_v}),\text{ and }\\
\mathbf{P}_{v,A}^2&=K_2[A,V]\setminus \mathbf{P}_v^1.
\end{align*}
\item If $c_{v,j_v}\in \calC^A_0$, let $\mathbf{P}_{v,A}^1=\emptyset$ and define
\begin{align*}
\mathbf{P}_{v,A}^0&=P_v^0\cap E(G^{good})\cap N_{\Gamma}(c_{v,j_v}),\text{ and }\\
\mathbf{P}_{v,A}^2&=K_2[A,V]\setminus \mathbf{P}_v^0.
\end{align*}
\item If $c_{v,j_v}\in \calC^A\setminus (\calC^A_0\cup \calC^A_1)$, let 
\begin{align*}
\mathbf{P}_{v,A}^0&=P_v^0\cap E(G^{good})\cap N_{\Gamma}(c_{v,j_v}),\\
\mathbf{P}_{v,A}^1&=P_v^1\cap E(G^{good})\cap N_{\Gamma}(c_{v,j_v}),\text{ and }\\
\mathbf{P}_{v,A}^2&=K_2[A,V]\setminus \Big(\mathbf{P}_{v,A}^1\cup \mathbf{P}_{v,A}^0\Big).
\end{align*}
\end{enumerate}
To show these definitions work to satisfy the conclusions of Claim \ref{cl:a}, we will need to compute the sizes of the edge colors above. For this, we will use the the following fact.  By definition of $P^{\tau}_v$ (see equality (\ref{pdef})) and hypothesis (\ref{sl:4}), we have that for all $x\in A$, $v\in \calS^A$ and $\tau\in \{0,1,2\}$,
\begin{align}\label{al:size}
|N_{P^{\tau}_{v}}(x)|=|N_{F_{\tau}}(q_v)||V|/\ell,
\end{align}
where $N_{F_{\tau}}(q_v)$ is computed in the auxiliary structure $\calG$.  We show that in all three cases, (a), (b), and (c)m conclusion (2) of Claim \ref{cl:a} holds.  Suppose first we are in case (a).  Then  by construction we have that for all $x\in A$,
\begin{align*}
|N_{\mathbf{P}_v^2}(x)|&\leq |V\setminus N_{\Gamma}(xc_{v,j_v})|+|N_{P^{bad}}(x)|+\sum_{p_u\in N_{F_0}(q_v)\cup N_{F_2}(q_v)}|N_{P_u}(x)|\\
&\leq |V\setminus N_{\Gamma}(xc_{v,j_v})|+|N_{P^{bad}}(x)|+|N_{F_0}(q_v)\cup N_{F_2}(q_v)||V|/\ell,
\end{align*}
where the last inequality is by (\ref{al:size}).  Since $A\subseteq Q_v^{good}(c_{v,j_v})$, we have that for all $x\in A$,
$$
|V\setminus N_{\Gamma}(xc_{v,j_v})|\leq 3\e^{1/4}|V|.
$$
Since $A\subseteq U_{good}'$, we have for all $x\in A$, $|N_{P^{bad}}(x)|\leq \e^{1/8} |V_{good}|$.  Since we are in case (a), $c_{v,j_v}\in \calC_A^1$ implies $|N_{F_0}(q_v)|\leq \e^{1/16}r$.  Hypothesis (\ref{sl:5}) implies $|N_{F_2}(q_v)|\leq \e r$.  Combining these facts, we have that
\begin{align*}
|N_{\mathbf{P}_v^2}(x)|&\leq  3\e^{1/4}|V|+\e^{1/8}|V_{good}|+\e r |V|/\ell +\e^{1/16}r|V|/\ell\leq 4\e^{1/16}|V|,
\end{align*}
where the last inequality uses that $r\leq \ell$.  A similar computation shows that in case (b), we have for all $x\in A$, $|N_{\mathbf{P}_v^2}(x)|\leq 4\e^{1/16}|V|$.  Suppose now we are in case (c).  Then via a similar argument to the above, we have that for all $x\in A$,
\begin{align*}
|N_{\mathbf{P}_{v,A}^2}(x)|\leq |V\setminus N_{\Gamma}(xc_{v,j_v})|+|N_{P^{bad}}(x)|+|N_{F_2}(q_v)||V|/\ell&\leq 3\e^{1/4}|V|+\e^{1/8}|V|+\e r|V|/\ell\\
&\leq 4\e^{1/16}|V|.
\end{align*}
This finishes our verification of conclusion (2) of Claim \ref{cl:a}. 

We next verify conclusion (3) of Claim \ref{cl:a}.  Let $v\in \calS^A$.  We want to show  there is no $\mathbf{P}_{v,A}^0/\mathbf{P}_{v,A}^1$-copy of $U(k)$ in $(A\cup V, \mathbf{P}_{v,A}^0,\mathbf{P}_{v,A}^1,\mathbf{P}_{v,A}^2)$. If we are in cases (a) or (b), then $c_{v,j_v}\in \calC^A_0\cup \calC^A_1$ implies one of $\mathbf{P}_{v,A}^0$ or $\mathbf{P}_{v,A}^1$ is empty, so there is clearly no $\mathbf{P}_{v,A}^0/\mathbf{P}_{v,A}^1$ copy of $U(k)$ in $(A\cup V, \mathbf{P}_{v,A}^0,\mathbf{P}_{v,A}^1,\mathbf{P}_{v,A}^2)$.  

Assume now we are in case (c), so $c_{v,j_v}\in \calC^A\setminus (\calC^A_0\cup \calC^A_1)$.  Suppose towards a contradiction there is a $\mathbf{P}_{v,A}^0/\mathbf{P}_{v,A}^1$ copy of $U(k)$ in $(A\cup V, \mathbf{P}_{v,A}^0,\mathbf{P}_{v,A}^1,\mathbf{P}_{v,A}^2)$.  This means there exist vertices $\{x_1,\ldots, x_{k}\}\cup \{y_S: S\subseteq [k]\}\subseteq A\cup V$ such that  for all $1\leq j\leq k$ and $S\subseteq [k]$, if $j\in S$, then $x_jy_S\in P_v^{1}\cap N_{\Gamma}(c_{v,j_v})$ and if $j\notin S$, then $x_jy_S\in P_{v}^{0}\cap N_{\Gamma}(c_{v,j_v})$.  However, since $A\subseteq N_{Q_v}(c_{v,j_v})$, this contradicts Observation \ref{ob}.  This finishes our verification of conclusion (3) of Claim \ref{cl:a}.

We just have left to show conclusion (4) of Claim \ref{cl:a}.  For each $v\in\calS^A$  and $\tau\in \{0,1,2\}$ define the triad
\begin{align*}
T(\mathbf{P}_{v,A}^{\tau},\mathbf{Q}^A_v)=(A\cup V\cup W,  \mathbf{P}_{v,A}^{\tau}\cup \mathbf{Q}_v^A\cup K_2[V,W]).
\end{align*}
Fix $v\in \calS^A$ and $\tau\in \{0,1\}$, and assume $\mathbf{P}_{v,A}^{\tau}\neq \emptyset$.  We want to show $T(\mathbf{P}_{v,A}^{\tau},\mathbf{Q}^A_v)$ is $2\e^{1/16}$-homogeneous with respect to $H$.    Since $A\subseteq Q_v^{good}(c_{v,j_v})$, we have that for all $x\in A$, $xc_{v,j_v}\in E(G^{good})$, and consequently,
$$
|N_{\Gamma}(c_{v,j_v})\cap K_2[\{x\},V]|\geq (1-\e^{1/4})|V|.
$$
Combining this with the fact $A\subseteq U_{good}'$, we have that for all $x\in A$,
\begin{align}\label{lb1}
|N_{{\bf P}_{v,A}^{\tau}}(x)|\geq |N_{P_{v,A}^{\tau}}(x)|-|N_{P^{bad}}(x)|-|V\setminus N_{\Gamma}(xc_{v,j_v})|&\geq |N_{F_{\tau}}(q_v)||V|/\ell-\e^{1/8}|V_{good}|-\e^{1/4}|V|\nonumber\\
&\geq \frac{\e^{1/16}r|V|}{\ell}-3\e^{1/8}|V|\nonumber\\
&\geq \e^{1/16}|V|/2,
\end{align}
where the  inequalities use (\ref{al:size}), the fact that $\mathbf{P}_{v,A}^{\tau}\neq \emptyset$ implies $|N_{F_{\tau}}(q_v)|\geq \e^{1/16}r$, and the fact that hypothesis (\ref{sl:2}) implies $r\geq (1-\e)\ell$.  We now have the following, where we recall $E(H)^1=E(H)$ and $E(H)^0={V(H)\choose 3}\setminus E(H)$.
\begin{align*}
|E(H)^{\tau}\cap K_3^{(2)}(T(\mathbf{P}_{v,A}^{\tau},\mathbf{Q}^A_v))|&\geq \sum_{x\in A} |N_{\Gamma}(x)\cap K_2[N_{\mathbf{P}_{v,A}^{\tau}}(x),N_{\mathbf{Q}^A_v}(x)]|\\
&=\sum_{x\in A}\Big(\sum_{w\in N_{\mathbf{Q}_v^A}(x)}|N_{\Gamma}(xw)\cap N_{\mathbf{P}_{v,A}^{\tau}}(x)|\Big) \\
&\geq \sum_{x\in A}\Big(\sum_{w\in N_{\mathbf{Q}_v^A}(x)}(|N_{\mathbf{P}_{v,A}^{\tau}}(x)|-|V\setminus N_{\Gamma}(xw)|)\Big)\\
&\geq \sum_{x\in A}\Big(\sum_{w\in N_{\mathbf{Q}_v^A}(x)}(|N_{\mathbf{P}_{v,A}^{\tau}}(x)|- \e^{1/4}|V|)\Big),
\end{align*}
where the last inequality uses $\mathbf{Q}_v^A\subseteq Q_v^{good}$.  Using (\ref{lb1}), this is at least 
\begin{align*}
\sum_{x\in A}\sum_{w\in N_{\mathbf{Q}_v^A}(x)}|N_{\mathbf{P}_{v,A}^{\tau}}(x)|(1- 2\e^{4})&=(1-2\e^{4})\sum_{x\in A}|N_{\mathbf{Q}_v^A}(x)||N_{\mathbf{P}_{v,A}^{\tau}}(x)|\\
&=(1-2\e^{4})|K_3^{(2)}(T(\mathbf{P}_{v,A}^{\tau},\mathbf{Q}^A_v))|,
\end{align*}
where the last inequality uses (\ref{lb1}).   This finishes our verification of conclusion (4), and thus we have finished proving Claim \ref{cl:a}. 
\end{proof}

\vspace{2mm}

\noindent{\bf Step IV.} In this step, we will apply Proposition \ref{prop:vcremoval} to the edge colored graphs constructed in Claim \ref{cl:a}.  Indeed, for each $A\in \calX\setminus \calX_{err}$, by Claim \ref{cl:a}, we can apply  Proposition \ref{prop:vcremoval} with $\e'=4\e^{1/16}$ and $\delta'=8(8c^4\e')^{1/4k+4}$.  This yields a partition  $\calP_{A,v}^V$ of $V$ and $\calP_{A,v}^A$ of $A$ so that the following hold, where $\mu=8(16 \e^{1/16}c^4)^{1/4k+4}$.
\begin{itemize}
\item $\max\{|\calP_{A,v}^V|,|\calP_{A,v}^A|\}\leq 2c(\mu/8)^{-k}$,
\item There are sets $\Sigma^0_{A,v},\Sigma^1_{A,v}\subseteq \calP_{A,v}^A\times \calP_{A,v}^V$ so that the following hold.
\begin{itemize}
\item $\bigcup_{(X,Y)\in \Sigma^0_{A,v}\cup \Sigma^1_{A,v}}(X\times Y)]|\geq (1-2\mu^{1/16})|A||V|$ 
\item For each $\tau\in \{0,1\}$ and $(X,Y)\in \Sigma^{\tau}_{A,v}$,  
$$
|\mathbf{P}_{v,A}^{\tau}\cap K_2[X,Y]| \geq (1-2\mu^{1/16})|X||Y|.
$$
\end{itemize}
\end{itemize}

\vspace{2mm}

\noindent{\bf Step V.} 
In this step, we will define a $(A,V,W)$-tripartite decomposition $\calD^A$ of $A\cup V\cup W$ for each $A\in \calX\setminus \calX_{err}$, which we will then show is homogeneous with respect to $H[A,V,W]$ in the sense of Definition \ref{def:triphom}.  Fix $A\in \calX\setminus \calX_{err}$.  We begin by defining the tripartite $(A,V,W)$-decomposition $\calD^A$ of $A\cup V\cup W$.  First, let 
$$
\calD^A_1=(\calD_{A}^A,\calD_A^V,\calD_A^W),
$$
where $D^A_W=\{W\}$, and $\calD^A_A$ and $\calD^A_V$ are the partitions of $A$ and $V$ obtained by taking the common refinements of the partitions $\calP_{v,A}^A$ and $\calP_{v,A}^V$, respectively, as $v$ ranges over $\calS^A$.   Observe that 
$$
|\calD_V^{A}|\leq \prod_{v\in \calS^A}|\calP_{v,A}^V|\leq (2c(\mu/8)^{-k})^{\ell}.
$$
We now define $\calD_2^A$.  For $X\in \calD_A^A$ and $Y\in \calD_V^A$, let $\calD_{XY}^A= \{K_2[X,Y]\}$.  For $Y\in \calD_V^A$, let $\calD_{YW}^A=\{K_2[Y,W]\}$.  Finally, for each $X\in \calD_A^A$, define 
$$
\calD^A_{XW}=\{\mathbf{Q}^A_v\cap K_2[X,W]: v\in \{0\}\cup \calS^A\}.
$$
We then let $\calD_2^A=\{\calD_{XY}^A, \calD_{YW}^A, \calD_{XW}^A: X\in \calD_A^A, Y\in \calD_V^A\}$.   We have now defined a tripartite $(A,V,W)$-decomposition $\calD^A=(\calD^A_1,\calD^A_2)$. 

Our next goal is to show $\calD^A$ is $3\mu^{1/32}$-homogeneous with respect to $H[A,V,W]$, in the sense of Definition \ref{def:triphom}.  The triads of $\calD^A$ have the following form, for some $X\in \calD_A^A$, $Y\in \calD_A^V$, and $v\in \{0\}\cup \calS^A$:
$$
T_{A,v}(X,Y):=(X\cup Y\cup W, K_2[Y,W]\cup (\mathbf{Q}_v^A\cap K_2[A,W])\cup K_2[X,Y]).
$$
Note that by construction, this gives us a partition
$$
K_3[A,V,W]=\bigcup_{v\in \{0\}\cup \calS^A}\bigcup_{X\in \calP_{A,v}^A}\bigcup_{Y\in \calP_{A,v}^V}K^{(2)}_3(T_{A,v}(X,Y)).
$$
 We show that most triples in $K_3[A,V,W]$ are covered by a \emph{homogeneous} triad of the form $T_{A,v}(X,Y)$.  The rough idea is to show that most such triads inherit  homogeneity from a triad of the form $T(\mathbf{P}_{v,A}^{\tau},\mathbf{Q}_v^A)$, for some $v\in \calS^A$ and $\tau\in \{0,1\}$ (recall from Claim \ref{cl:a} that all such triads are $2\e^{1/16}$-homogeneous with respect to $H$).

Our next several definitions are aimed at defining a set of triads from $\calD^A$ which we will show are  homogeneous.  First, define
$$
\Gamma_A=\bigcup_{v\in \calS^A}\bigcup_{\tau\in \{0,1\}}\Big(E(H)^{\tau}\cap K_3^{(2)}(T(\mathbf{P}_{v,A}^{\tau},\mathbf{Q}^A_v))\Big).
$$
In other words, $\Gamma_A$ is the set of triples $xyz\in K_3[A,V,W]$ which come from a homogeneous triad of the form $T(\mathbf{P}_{v,A}^{\tau},\mathbf{Q}_v^A)$, for some $v\in \calS^A$ and $\tau\in \{0,1\}$, and so that $xyz$ is  ``correct" with respect to this triad.  We now define the set of triads from $\calD^A$ which are mostly covered by $\Gamma_A$.
$$
\calT_A=\{G\in Tri(\calD^A): |\Gamma_A\cap K_3^{(2)}(G)|\geq (1-10\e^{1/32})|K_3^{(2)}(G)|\}.
$$
This set will be useful because, if $G\in \calT_A$ and $G$ is mostly contained in a \emph{single} triad of the form  $T(\mathbf{P}_{v,A}^{\tau},\mathbf{Q}_v^A)$ with $v\in \calS^A$ and $\tau\in \{0,1\}$, then we can show $G$ is also somewhat homogeneous. We now work towards making this precise.  First, given $v\in \calS^A$, define 
\begin{align*}
\Omega_{A,v}^1&=\{(X,Y)\in \calD_A^A\times \calD_V^A: |\mathbf{P}_{v,A}^{1}\cap K_2[X,Y]| \geq (1-\mu^{1/32})|X||Y|\}\text{ and }\\
\Omega_{A,v}^0&=\{(X,Y)\in \calD_A^A\times \calD^A_V: |\mathbf{P}_{v,A}^{0}\cap K_2[X,Y]| \geq (1-\mu^{1/32})|X||Y|\}.
\end{align*}
Since $\calP^A_{v,A}\cup \calP_{v,A}^V$ is a $2\mu^{1/16}$-homogeneous partition of $(A\cup V,{\bf P}_v^0,{\bf P}_v^1, {\bf P}_v^2)$ (see Definition \ref{def:echom}), we have by Lemma \ref{lem:averaging} and the definitions of $\Omega^0_{A,v}, \Omega^1_{A,v}$ that for each $v\in \calS^A$, 
\begin{align}\label{al:omega}
|\bigcup_{(X,Y)\in \Omega^0_{A,v}\cup \Omega^1_{A,v}}K_2[X,Y]|\geq (1-2\mu^{1/32})|A||V|.
\end{align}
The sets $\Omega_{A,v}^0,\Omega_{A,v}^1$ are important because, given a triad of the form $T_{A,v}(X,Y)$, if we know $(X,Y)\in \Omega_{A,v}^{\tau}$ for some $\tau\in \{0,1\}$, then $T_{A,v}(X,Y)$ is mostly contained in the homogeneous triad $T(\mathbf{P}_{v,A}^{\tau},\mathbf{Q}_v^A)$ (we will make this precise shortly).  We now give a name to the triads of the form $T_{v,A}(X,Y)$ where $v\in \calS^A$ and $(X,Y)\in \Omega_{A,v}^0\cup \Omega_{A,v}^1$.
$$
\calT_A'=\{T_{A,v}(X,Y):  v\in \calS^A, (X,Y)\in \Omega_{A,v}^0\cup \Omega_{A,v}^1\}.
$$
We will also need a name for the following set of triads which are not too small.  
$$
\calT_A''=\{T\in \triads(\calD^A): |K_3^{(2)}(T)|\geq (1-\e^{1/32})|X||Y||W|/\ell\}.
$$
Finally, we define 
$$
\calT_A^{good}=\calT_A\cap \calT_A'\cap \calT_A''.
$$
We will show in Claim \ref{cl:ahom} below that the triads in $\calT_A^{good}$ are homogeneous with with respect to $H[A,V,W]$.  

\begin{claim}\label{cl:ahom}
Every $T\in \calT_A^{good}$ is $3\mu^{1/32}$-homogeneous with respect to $H[A,V,W]$.
\end{claim}
\begin{proof}
Fix $T\in \calT_A^{good}$.  Then $T$ has the form $T_{A,v}(X,Y)$ some $v\in \calS^A$ and $(X,Y)\in  \Omega_{A,v}^0\cup \Omega_{A,v}^1$.  Without loss of generality, say $(X,Y)\in   \Omega_{A,v}^1$ (the other case is similar).  Observe that
\begin{align*}
|E(H)\cap K_3^{(2)}(T)|&\geq |K_3^{(2)}(T)\cap K_3^{(2)}(T_{A,v}(\mathbf{P}_{v,A}^1, \mathbf{Q}_v^A))\cap \Gamma_A|\\
&\geq  |K_3^{(2)}(T)|-|K_3^{(2)}(T)\setminus \Gamma_A|-|K_3^{(2)}(T)\setminus K_3^{(2)}(T(\mathbf{P}_{v,A}^1, \mathbf{Q}_v^A))|.
\end{align*}
By definition of $\calT_A^{good}$, we know that $|K_3^{(2)}(T)\setminus \Gamma_A|\leq \e^{1/32}|K_3^{(2)}(T)|$.  Using that $(X,Y)\in \Omega^1_{A,v}$, hypothesis (\ref{sl:4}) and the fact that $\mathbf{Q}_v^A\subseteq Q_v$, we have
\begin{align*}
|K_3^{(2)}(G)\setminus K_3^{(2)}(T_{A,v}(\mathbf{P}_{v,A}^1, \mathbf{Q}_v^A))|&=\sum_{x\in X}|N_{\mathbf{Q}_v^A}(x)||Y\setminus N_{\mathbf{P}_{v,A}^1}(x)|\\
&\leq \sum_{x\in X}(|W|/\ell)|Y\setminus N_{\mathbf{P}_{v,A}^1}(x)|\\
&= (|W|/\ell)|K_2[X,Y]\setminus \mathbf{P}_{v,A}^1|\\
&\leq \e^{1/32}|X||Y||W|/\ell\\
&\leq 2\e^{1/32}|K_3^{(2)}(T)|,
\end{align*}
where the last inequality is since $T\in \calT_A^{good}$ implies $|K_3^{(2)}(T)|\geq (1-\e^{1/32})|X||Y||W|/\ell$.  Combining, we have that $|E(H)\cap K_3^{(2)}(T)|\geq (1-3\e^{1/32})|K_3^{(2)}(T)|$, as desired.
\end{proof}

We now show that most triples in $K_3[A,V,W]$ come from a triad in $\calT_A^{good}$.

\begin{claim}\label{cl:ahomcover}
We have $|\bigcup_{T\in \calT_A^{good}}K_3^{(2)}(T)|\geq (1-3\mu^{1/32})|A||V||W|$.
\end{claim}
\begin{proof}
We first bound the number of triples from $K_3[A,V,W]$ not in a $\calT_A$ triad.  By Claim \ref{cl:a} parts (1) and (2), and since the $\mathbf{Q}_v^A$ are pairwise disjoint, we have
\begin{align*}
|K_3[A,V,W]\setminus \Gamma_A|&\leq |\mathbf{Q}_0^A||V|+ \sum_{v\in \calS^A}\sum_{\tau\in \{0,1\}}|E(H)^{\tau}\setminus K_3^{(2)}(T(\mathbf{P}_{v,A}^{\tau}, \mathbf{Q}^A_v))|\\
&+\sum_{v\in \calS^A}|K_3^{(2)}(T(\mathbf{P}_{v,A}^2, \mathbf{Q}^A_v))|\\
&\leq 5\e^{1/8}|A||V||W|+\sum_{v\in \calS^A}\sum_{\tau\in \{0,1\}}2\e^{1/16}|K_3^{(2)}(T(\mathbf{P}_{v,A}^{\tau}, \mathbf{Q}^A_v))|\\
&+\sum_{v\in \calS^A}|K_3^{(2)}(T(\mathbf{P}_{v,A}^2, \mathbf{Q}^A_v))|\\
&\leq 5\e^{1/8}|A||V||W|+2\e^{1/16}|V|\sum_{v\in \calS^A}|\mathbf{Q}_v^A|+\sum_{x\in A}\sum_{v\in \calS^A}|N_{\mathbf{P}_{v,A}^{2}}(x)||N_{\mathbf{Q}_v^A}(x)|\\
&\leq 5\e^{1/8}|A||V||W|+2\e^{1/16}|A||V||W|+\sum_{x\in A}\sum_{v\in \calS^A}4\e^{1/16}|V||N_{\mathbf{Q}_v^A}(x)|\\
&\leq 5\e^{1/8}|A||V||W|+2\e^{1/16}|A||V||W|+4\e^{1/16}|A||V||W|\\
&\leq 11\e^{1/16}|A||V||W|.
\end{align*}
Consequently, by Lemma \ref{lem:averaging} we have
$$
\Big|\bigcup_{T\in \calT_A}K_3^{(2)}(T)\Big|\geq (1-11\e^{1/32})|A||V||W|.
$$
We now consider triads from $\calT_A'$.  By (\ref{al:omega}), 
\begin{align*}
|K_3[A,V,W]\setminus \Big(\bigcup_{T\in \calT_A'}K_3^{(2)}(T)\Big)|&\leq |\mathbf{Q}_0^A||V|+\sum_{v\in \calS^A}\Big(\sum_{\{(X,Y)\in (\calP_U\times \calP_V)\setminus (\Omega^0_{A,v}\cup \Omega^1_{A,v})\}}\Big(\sum_{x\in X}|Y||N_{\mathbf{Q}_v^A}(x)|\Big)\Big)\\
&\leq \frac{|W|}{\ell}\sum_{v\in \calS^A}\Big(\sum_{\{(X,Y)\in (\calP_U\times \calP_V)\setminus (\Omega^0_{A,v}\cup \Omega^1_{A,v})\}}|X||Y|\Big)\\
&\leq 2\mu^{1/32}|W||A||V|.
\end{align*}
We now turn to $\calT_A''$. Given $X\in \calD_A^A$, define 
$$
\calS^A_{small}(X)=\{v\in \calS^A:  |\mathbf{Q}_v^A\cap K_2[X,W]|\leq  (1-\e^{1/32})|X||W|/\ell\}.
$$
For $v\in \calS^A$ and $Y\in \calD^A_V$, we observe $T_{A,v}(X,Y)\notin \calT_A''$ if and only if $v\in \calS^A_{small}(X)$, since by definition,
$$
|K_3^{(2)}(T_{A,v}(X,Y))=|\mathbf{Q}_v^A\cap K_2[X,W]||Y|.
$$
By Claim \ref{cl:a}, we have 
\begin{align}\label{al:kq}
(1-5\e^{1/8})|X||W|\leq |K_2[X,W]\setminus \mathbf{Q}_0^A|.
\end{align}
By definition of $\calS_A^{small}(X)$, we have
\begin{align*}
|K_2[X,W]\setminus \mathbf{Q}_0^A|&\leq |\calS^A_{small}|(1-\e^{1/32})|X||W|/\ell+|\calS^A\setminus \calS^A_{small}||X||W|/\ell\\
&=\frac{|X||W|}{\ell}(|\calS^A|-\e^{1/32}|\calS^A_{small}|)\\
&\leq |X||W|(1-\e^{1/32}|\calS^A_{small}|/\ell),
\end{align*}
where the last inequality uses that $|\calS^A|\leq \ell$.  Combining with (\ref{al:kq}), we have that $ |X||W|(1-\e^{1/32} |\calS^A_{small}|/\ell)\geq (1-5\e^{1/8})|X||W|$, so 
$$
|\calS^A_{small}|\leq 5\e^{1/8-1/32}\ell\leq \e^{1/16}\ell.
$$
Thus  
\begin{align*}
|K_3[A,V,W]\setminus \Big(\bigcup_{T\in \calT_A''}K_3^{(2)}(T)\Big)|&\leq \sum_{X\in \calD_A^A}\sum_{v\in \calS^A_{small}(X)}|\mathbf{Q}_v^A\cap K_2[X,W]||V|\\
&\leq \sum_{X\in \calD_A^A}\sum_{v\in \calS^A_{small}(X)}|X||V||W|/\ell\\
&\leq \e^{1/16}|A||V||W|.
\end{align*}
We now have that
\begin{align*}
&|K_3[A,V,W]\setminus\Big( \bigcup_{T\in \calT_A^{good}}K_3^{(2)}(T)\Big)|&\leq (11\e^{1/16}+2\mu^{1/32}+\e^{1/16})|A||V||W|\leq 3\mu^{1/32}|A||V||W|.
\end{align*}
This finishes the proof of Claim \ref{cl:ahomcover}.
\end{proof}

\vspace{2mm}

\noindent{\bf Step VI.} We show that for all $A\in \calX\setminus \calX_{err}$, the partition $\calD_A^V$ arising from Step V is almost $6\mu^{1/128}$-good with respect to $H[A,V,W]$ (recall Definition \ref{def:almostgood}).  First, define
\begin{align*}
\calV_{A}=\{Y\in \calD_{V}^A: (\bigcup_{T\in \calT_A^{good}}K_3^{(2)}(T))\cap K_3[A,Y,W]|\geq (1-3\mu^{1/64})|A||Y||W|\}.
\end{align*}
Now set $\calG_{A}=\bigcup_{Y\in \calV_{A}}Y$.  By Claim \ref{cl:ahomcover} and Lemma \ref{lem:averaging}, $|\calG_{A}|\geq 3\mu^{1/64}|V|$.  We show every $Y\in \calV_{A}$ is almost $6\mu^{1/128}$-good with respect to $H[A,V,W]$. Given $T_{A,v}(X,Y)\in \calT_A^{good}$, let 
$$
B(T_{A,v}(X,Y))=\{xw\in \mathbf{Q}^A_v\cap K_2[X,W]: \frac{|N_H(xw)\cap Y|}{|Y|}\in [0,3\mu^{1/64})\cup (1-3\mu^{1/64},1]\}.
$$
Since any $T_{A,v}(X,Y)\in \calT_A^{good}$ is $3\mu^{1/32}$-homogeneous with respect to $H[A,V,W]$, a standard counting argument implies 
$$
|B(T_{A,v}(X,Y))|\geq (1-3\mu^{1/64})|\mathbf{Q}^A_v\cap K_2[X,W]|.
$$
Given $Y\in \calD_{V}^A$, let
$$
\calS(Y)=\{xw\in K_2[X,W]: |N_H(xw)\cap Y|/|Y|\in [0,3\mu^{1/128})\cup (1-3\mu^{1/128},1]\}.
$$
For every $Y\in \calD_{V}^A$ and $v\in \calS^A$, we have 
$$
|\calS(Y)\cap \mathbf{Q}_v^A|\geq |\bigcup_{T_{A,v}(X,Y)\in \calT_A^{good}}B(T_{A,v}(X,Y))|\geq (1-3\mu^{1/64})|\bigcup_{T_{A,v}(X,Y)\in \calT_A^{good}}\mathbf{Q}_v^A\cap K_2[X,W]|.
$$
Thus, for every $Y\in \calD_V^A$, 
\begin{align*}
|\calS(Y)|\geq\sum_{v\in \calS^A} |\bigcup_{T_{A,v}(X,Y)\in \calT_A^{good}}B(T_{A,v}(X,Y))|&\geq (1-3\mu^{1/64})\sum_{v\in \calS^A}|\bigcup_{T_{A,v}(X,Y)\in \calT_A^{good}}\mathbf{Q}_v^A\cap K_2[X,W]|.
\end{align*}
If $Y\in \calV_A$, then Lemma \ref{lem:averaging} implies 
$$
|\bigcup_{v\in \calS^A}\bigcup_{T_{A,v}(X,Y)\in \calT_A^{good}}\mathbf{Q}_v^A\cap K_2[X,W]|\geq (1-3\mu^{1/128})|A||W|.
$$
Thus, for every $Y\in \calV_A$, $|\calS(Y)|\geq (1-3\mu^{1/128})^2|A||W|$.  This shows that every $Y\in \calV_{A}$ is almost $6\mu^{1/128}$-good with respect to $H[A,V,W]$.  Since $|\calG_{A}|\geq 3\mu^{1/64}|V|$, this finishes our verfication that $\calD_A^V$ is almost $6\mu^{1/128}$-good with respect to $H[A,V,W]$. 

\vspace{2mm}

\noindent{\bf Step VII.} Define $\calP_V$ to be the common refinement of the $\calD_V^A$ as $A$ ranges over the elements in $\calX\setminus \calX_{err}$.  Note, using (\ref{pui}),
$$
|\calP_V|\leq  ((2c(\mu/8)^{-k})^{\ell})^{|\calX|}\leq ((2c(\mu/8)^{-k})^{\ell})^{(2\ell)^{\ell}\e^{-\ell}}=(2c(\mu/8)^{-k})^{\ell(2\ell)^{\ell}\e^{-\ell}}.
$$
By Lemma \ref{lem:refinement}, $\calP_V$ is $6\mu^{1/512}$-homogeneous with respect to $H[A,V,W]$ for every $A\in \calX\setminus \calX_{err}$.  Since 
$$
|V\setminus (\bigcup_{A\in \calX\setminus \calX_{err}}A)|\leq |U_{err}|\leq 3\e^{1/8}<\mu^{1/512},
$$
this implies $\calP_V$ is almost $6\mu^{1/512}$-good with respect to $H$. This concludes the proof of Proposition \ref{prop:mainslice}.
\qed

\vspace{2mm}

\subsection{Proof of Theorem}\label{ss:final}

In this section, we prove  Theorem \ref{thm:slvc}.

\vspace{2mm}

\noindent{\bf Proof of Theorem \ref{thm:slvc}.} Fix $k\geq 1$, and let $c_1=c_1(k)$ be from Proposition \ref{prop:vcremoval}.   We will use the following three related constants.
\begin{align*}
c_2=16c_1^4,\text{ }c_3=c_2^{1/(4k+4)(100000)}/2\text{ and }c_4=(4k+4)(100000)(16).
\end{align*}
Let $C_1, D_1$ be from Theorem \ref{thm:bipfoxpachsuk} applied with parameter $k$. For convenience, we will work with the constant $K_1=\max\{2,C_1,D_1\}$. We then set 
$$
K_2=300K_1^2c_3^{-200K_1},\text{ }K_3=200c_4K_1,\text{ and }K=K_2+K_3.
$$
We note here that the constants $c_1,c_2,c_3,c_4,K_1,K_2,K_3,K$ all depend only on $k$.

Now fix $\tau>0$ sufficiently small compared to all the parameters above.  We define  two additional ``small" parameters:
$$
\e=(c_3\tau)^{c_4}\text{ and }\mu=8(c_2\e^{1/16})^{1/(4k+4)}.
$$
The reader should think of $\e$ as being much smaller than $\mu$, and $\mu$ as being much smaller than $\tau$.  Note that in particular $\mu\leq 8(16\e^{1/16}c_1^4)^{1/4k+4}$ (this will be needed later for an application of Proposition \ref{prop:mainslice}).   We then set 
$$
\ell=K_1\e^{-100K_1}.
$$

Suppose $n$ is sufficiently large, and $H=(U\cup V\cup W, F)$ is a tripartite $3$-graph with $|U|=|V|=|W|=n$, and slicewise VC-dimension at most $k$.  We will show $H$ has a $\tau$-homogeneous partition of size at most $2^{2^{\tau^{-K}}}$.  The first main ingredient in the proof is the following claim.

\begin{claim}\label{cl:amalgamateU}
There exists a partition $\calP_V$ of $V$ which is almost $6\mu^{1/512}$-good with respect to $H$ and so that 
$$
|\calP_V|\leq (2c_1(\mu/8)^{-k})^{3^{\ell^2}\ell(2\ell)^{\ell}\e^{-\ell/8}}.
$$
\end{claim}

We now give a sketch for how we will prove Claim \ref{cl:amalgamateU}.  We start by obtaining regular partitions $\calP_x$ for all the slice graphs $H_x$ for $x\in U$. We then group together elements $x\in U$ based on the auxiliary structures arising from the regular partition $\calP_x$ of $H_x$.  This will set us up to apply Proposition \ref{prop:mainslice} on $H$, restricted to each of these equivalence classes. We will then take the common refinements of the resulting partitions of $V$, and show this yields the conclusion of Claim \ref{cl:amalgamateU}.

\vspace{2mm}

\noindent{\bf Proof of Claim \ref{cl:amalgamateU}.} Fix $x\in U$. We define several auxiliary objects associated to the slice-graph $H_x$.  By Theorem \ref{thm:bipfoxpachsuk}, there exists an $\e^{100}$-homogeneous equipartition $\calP_x=\{V_x^1,\ldots, V^{\ell_x}_x, W_x^1,\ldots, W^{\ell_x}_x\}$ of $H_x$,   for some $\ell_x\leq \ell$, where $V=\bigcup_{i=1}^{\ell_x}V_x^i$ and $W=\bigcup_{i=1}^{\ell_x}W_x^i$.  We emphasize to the reader that $\calP_x$ is an \emph{equipartition}.  Let $Irr(\calP_x)$ be the set of pairs $(V_x^i,W^j_x)$ which are not $\e^{100}$-homogeneous for $H_x$. By assumption, and since $\calP_x$ is an equipartition, we have $|Irr(\calP_x)|\leq \e^{100}\ell_x^2$.  Define
\begin{align*}
\calP_x^{err,V}&=\{V_i^x: |\{W^j_x: (V^i_x,W^j_x)\in Irr(\calP_x)\}|\geq \e^{50}\ell_x\}\text{ and }\\
\calP_x^{err,W}&=\{W_i^x: |\{V^j_x: (V^j_x,W^i_x)\in Irr(\calP_x)\}|\geq \e^{50}\ell_x\}.
\end{align*}
Now let $B^0_x=\bigcup_{X\in \calP_x^{err,V}}X$ and $C^0_x=\bigcup_{X\in \calP_x^{err,W}}X$.  Clearly 
$$
|B^0_x|\leq \e^{50}|V|\text{ and }|C^0_x|\leq \e^{50}|W|.
$$
Fix enumerations
 \begin{align*}
 \{B_1^x,\ldots, B_{r_x}^x\}&=\{V_1^x,\ldots, V_{\ell_x}^x\}\setminus \calP_x^{err,V}\text{ and }\\
\{C_1^x,\ldots, C_{s_x}^x\}&=\{W_1^x,\ldots, W_{\ell_x}^x\}\setminus \calP_x^{err,W}.
\end{align*}
 We define an auxiliary edge-colored graph $\calG_x$ as follows.  Let $\calG_x$ have vertex set 
 $$
 \{b_0,c_0\}\cup \{b_i:i\in [r_x] \}\cup \{c_i: i\in [s_x] \}
 $$
  and edge colors $F^1_x,F_x^0,F_x^2$, where
\begin{align*}
F^1_x&=\{b_ic_j: i\in [r_x], j\in [s_x], |E(H_x)\cap K_2[B^i_x,C^j_x]|\geq (1-\e^{100})|B^i_x||C^j_x|\}\\
F^0_x&=\{b_ic_j: i\in [r_x], j\in [s_x], |E(H_x)\cap K_2[B^i_x,C^j_x]| \leq \e^{100} |B^i_x||C^j_x|\}\\
F^2_x&=\{b_ic_j: i=0, j=0,\text{ or }\e^{100} |B^i_x||C^j_x|\leq |E(H_x)\cap K_2[B^i_x,C^j_x]| \leq (1-\e^{100}) |B^i_x||C^j_x|\}.
\end{align*}
Clearly $|F^2_x|\leq 2\e^{50}\ell_x^2$.  Moreover, for all $1\leq i\leq r_x$ and $1\leq j\leq s_x$,
\begin{align}\label{gst}
|N_{F^2_x}(b_i)|&\leq \frac{|W|}{\ell_x}|\{w\in W: w\in W^j_x\text{ some} (V^i_x,W^j_x)\in Irr(\calP_x)\}|\leq \e^{50}\ell\text{ and }\nonumber\\
|N_{F^2_x}(c_j)|&\leq \frac{|V|}{\ell_x}|\{v\in V: v\in V^i_x\text{ some} (V^i_x,W^j_x)\in Irr(\calP_x)\}|\leq \e^{50}\ell.
\end{align}

Given $x,x'\in U$, write $x\equiv x'$ if and only if $\ell_x=\ell_{x'}$ and $\calG_x=\calG_{x'}$.  Clearly $\equiv$ is an equivalence relation.  Let $\calE$ be the set of equivalence classes of $\equiv$ in $U$.  Note $|\calE|\leq \ell3^{\ell^2}$.  To ease notation, given $X\in \calE$, let $\ell_X$, $r_X$, $s_X$, and $\calG_X$ be such that for all $x\in X$, $\ell_x=\ell_X$, $r_x=r_X$, $s_x=s_X$, and $\calG_x=\calG_X$.  By relabeling, we may assume $\calG_X$ has vertex set $\{p_0,\ldots,p_{r_X},q_0,\ldots, q_{s_X}\}$ and edge sets $F_{0}^X,F_{1}^X,F_{2}^X$.   Let 
$$
\calE_{err}=\{X\in \calE:  |X|<\e |V|/|\calE|\}\text{ and }U_{err}=\bigcup_{X\in \calE_{err}}X.
$$
Clearly $|U_{err}|\leq \e |U|$.  For each $X\in \calE\setminus \calE_{err}$, we apply Proposition \ref{prop:mainslice} to $\calG_X$ and $H[X,V,W]$ with parameters $r_X,s_X,\ell_X$, and $\e^{100}$.  This yields,  for all $X\in \calE\setminus \calE_{err}$, a partition $\calP^X_V$ of $V$ so that the following hold.
\begin{itemize}
\item $\calP_V^X$ is almost $6\mu^{1/512}$-good with respect to $H[X,V,W]$,
\item $\calP_V^X$ has size at most $(2c_1(\mu/8)^{-k})^{\ell(2\ell)^{\ell}\e^{-\ell}}$.
\end{itemize}
We now let $\calP_V$ be the common refinement of the $\calP_V^X$ as $X$ ranges over the elements of $\calE\setminus \calE_{err}$.  Note that since $|\calE|\leq \ell 3^{\ell^2}$, we have
$$
|\calP_V|\leq ((2c_1(\mu/8)^{-k})^{\ell(2\ell)^{\ell}\e^{-\ell}})^{|\calE|}\leq (2c_1(\mu/8)^{-k})^{3^{\ell^2}\ell^2(2\ell)^{\ell}\e^{-\ell}}.
$$
This ends the proof of Claim \ref{cl:amalgamateU}
\qed

\vspace{2mm}

We now finish the proof of the theorem.  By symmetry, analogues of Claim \ref{cl:amalgamateU} holds with the roles of $U,V,W$ permuted.  Say $\calP_U$, $\calP_V$, and $\calP_W$ are the resulting parititions of $U$, $V$, and $W$.  Let $\calP=\calP_U\cup \calP_V\cup \calP_W$.  Note
$$
|\calP|\leq 3(2c_1(\mu/8)^{-k})^{3^{\ell^2}\ell^2(2\ell)^{\ell}\e^{-\ell}}.
$$
By Proposition \ref{prop:goodhom}, $\calP$ is $28(6\mu^{1/512})^{1/64}$-homogeneous with respect to $H$.  Note 
$$
28(6\mu^{1/512})^{1/64}\leq \mu^{1/100000},
$$
so we have  $\calP$ is $\mu^{1/10000}$-homogeneous with respect to $H$. Recall we defined $\ell= K_1\e^{-100K_1}$, $\mu=8(c_2\e^{1/16})^{1/(4k+4)}$, and $\e=(c_3\tau)^{c_4}$.  Consequently, 
\begin{align*}
\mu^{1/10000}=(8(c_2\e^{1/16})^{1/(4k+4)})^{1/10000}&=(8(c_2(c_3\tau)^{c_4/16})^{1/(4k+4)})^{1/10000}\\
&=(8)^{1/10000}(c_2c_3^{c_4/16})^{1/((4k+4)(10000))}\tau^{c_4/((4k+4)(10000)(16))}\\
&\leq 2(c_2c_3^{c_4/16})^{1/((4k+4)(10000))}\tau^{c_4/((4k+4)(10000)(16))}\\
&=\tau,
\end{align*}
where the second to last inequality is because $8^{1/10000}<2$, and the last inequality is by definition of $c_2$, $c_3$ and $c_4$.  Therefore, $\calP$ is $\tau$-homogeneous.  We now consider the bound on the size of $\calP$.  By definition, $K_1\geq 2$, so we have $\ell=K_1\e^{-1000K_1} \geq \e^{-1}$.  Consequently, 
$$
3^{\ell^2}\ell^2(2\ell)^{\ell}\e^{-\ell}\leq (3^{\ell^2})^4=12^{\ell^2}.
$$
Thus, since we have $8\mu^{-1}<\e^{-1}$, $2c_1<\e^{-2}$, and  $\ell\geq \e^{-1}>k^{-1}$,
\begin{align*}
|\calP|\leq 3(2c_1(\mu/8)^{-k})^{12^{\ell^2}}\leq 3(2c_1\e^{-k})^{12^{\ell^2}}\leq 3(\e^{-2k})^{12^{\ell^2}}\leq \e^{-64^{\ell^2}}\leq 2^{2^{300 \ell^2}}.
\end{align*}
Since $\ell=K_1\e^{-100K_1}=K_1c_3^{-100K_1}\tau^{-100c_4K_1}$, we have 
$$
300\ell^2=300K_1^2\e^{-200K_1}=300K^2_1c_3^{-200K_1}\tau^{-200c_4K_1}=K_2\tau^{-K_3}\leq \tau^{-K_2-K_3}=\tau^{-K},
$$
where the last equalities are by definition of $K_2$, $K_3$, and $K$.  Combining with the above, this yields $|\calP|\leq 2^{2^{\tau^{-K}}}$, as desired.
\qed

\vspace{2mm}

\appendix

\section{}

We begin with a proof of Theorem \ref{thm:bipfoxpachsuk}.

\vspace{2mm}

\noindent{\bf Proof of Theorem \ref{thm:bipfoxpachsuk}.}. Let $c=c(k)$ be as in Theorem \ref{thm:foxpachsuk}. Set $D=2k+2$ and $C=2c$. Fix $\e>0$ and assume $H=(A\cup B, E)$ is a bipartite graph with VC-dimension less than $k$ with $|A|=|B|$. By Theorem \ref{thm:foxpachsuk}, there is an $\e$-homogeneous equipartition $\calP=\{V_1,\ldots, V_t\}$ of $V(H)$ with $8\e^{-1}\leq t\leq c\e^{-2k-1}$.  We begin by finding an equipartition of $\calP$ which refines $\{A,B\}$.  We start by defining, for each $i\in [t]$, let $A_i=A\cap V_i$ and $B_i=B\cap V_i$.  These sets will not all have the same size, so we will now collect together the ones which are too small, and split the remaining parts into equally sized pieces.  In particular, define
$$
\calA_0=\{A_i: i\in [t], |A_i|<\e |A|/t\}\text{ and }\calB_0=\{B_i: i\in [t], |B_i|<\e |B|/t\}.
$$
For each $A_i\notin \calA_0$, let $A_{i0},A_{i1},\ldots, A_{ij_i}$ be a partition of $A_i$ so that for each $1\leq u\leq j_i$, $|A_{iu}|=\lceil \e |A|/t\rceil$ and $|A_{i0}|\leq \e|A|/t$. Similarly, for each $B_i\notin \calB_0$, let $B_{i0},A_{i1},\ldots, B_{is_i}$ be a partition of $B_i$ so that for each $1\leq u\leq s_i$, $|B_{iu}|=\lceil \e |B|/t\rceil$ and $|B_{i0}|\leq \e|A|/t$.  Fix new enumerations 
\begin{align*}
\{A_1',\ldots, A_{t_1}'\}&=\{A_{iu}: A_i\notin \calA_0, 1\leq u\leq j_i\}\text{ and }\\
\{B_1',\ldots, B_{t_2}'\}&=\{B_{iu}: B_i\notin \calB_0, 1\leq u\leq s_i\}.
\end{align*}
Note $t_1,t_2\leq \e^{-1}t$.  Without loss of generality, assume $t_1\leq t_2$ (the other case is similar). Now define
\begin{align*}
A_0&=\Big(\bigcup_{A_i\in \calA_0}A_i\Big)\cup \Big(\bigcup_{A_i\notin \calA_0}A_{i0}\Big)\text{ and }B_0=\Big(\bigcup_{B_i\in \calB_0}B_i\Big)\cup \Big(\bigcup_{B_i\notin \calB_0}B_{i0}\Big)\cup \Big(\bigcup_{t_1<i\leq t}B_i'\Big).
\end{align*}
Clearly $|A_0|\leq \e |A|$.  Since each $|A_j'|$ and $|B_i'|$ have the same size, 
$$
|B_0|=|B|-t_1\lceil \e |B|/t\rceil = |A|-t_1\lceil \e |A|/t\rceil =|A_0|\leq \e |A|=\e|B|.
$$  
Now let $\calP'=\{A_0,B_0\}\cup \{A'_i: i\in [t_1]\}\cup \{B'_i: i\in [t_1]\}$.   Let $\Sigma_{reg}$ be the set of $\e$-homogeneous pairs from $\calP$, and let $\Sigma_{reg}'$ be the set of $\sqrt{\e}$-homogeneous pairs from $\calP'$.  By assumption, $\sum_{(X,Y)\in \Sigma_{reg}}|X||Y|\geq (1-\e)|V(H)|^2$.  Since $\calP'$ refines $\calP$, we have by Lemma \ref{lem:averaging} that for any $(X,Y)\in \Sigma_{reg}$,
$$
\Big| \bigcup_{\{(X',Y')\in \Sigma_{reg}': X'\subseteq X, Y'\subseteq Y\}}X'\times Y']\Big|\geq (1-\sqrt{\e})|X||Y|.
$$
Thus
$$
\Big| \bigcup_{(X',Y')\in \Sigma_{reg}'}X'\times Y'\Big|\geq (1-\sqrt{\e}) \Big| \bigcup_{(X,Y)\in \Sigma_{reg}}X\times Y\Big|\geq (1-\sqrt{\e})(1-\e)|V(H)|^2.
$$
Now choose arbitrary equipartitions $A_0=A_{01}\cup \ldots \cup A_{0t_1}$ and  $B_0=B_{01}\cup \ldots \cup B_{0t_1}$, and for each $i\in [t]$, let $A_i''=A'_i\cup A_{0i}$ and $B_i''=B_i'\cup B_{0i}$.  Note that for each $i\in [t_1]$, $|A_{0i}|<\e|A|/t_1=\e|A_i''|$ and $|B_{0i}|<\e|B|/t_1=\e|B_i''|$.  Consequently, if $(A_i',B_j')$ is in $\Sigma_{reg}'$, is is straightforward to show $(A_i'',B_j'')$ is $2\sqrt{\e}$-homogeneous. Finally, we set 
$$
\calP''=\{A_i'': i\in [t_1]\}\cup \{B_i'':i\in [t_1]\},
$$
and we let $\Sigma_{reg}''$ be the set of $2\sqrt{\e}$-homogeneous pairs from $\calP''$.  We then have
$$
\Big| \bigcup_{(X'',Y'')\in \Sigma_{reg}''}X'\times Y'\Big|\geq  \Big| \bigcup_{(X',Y')\in \Sigma_{reg}'}X'\times Y'\Big|-|A_0||B|-|B_0||B|\geq (1-2\sqrt{\e})|X||Y|.
$$
We have now constructed $\calP''$, a $2\sqrt{\e}$-homogeneous equipartition of $A\cup B$ refining $\{A,B\}$.  Note $|\calP''|=2t_1\leq 2\e^{-1}t\leq 2c\e^{-2k-2}=C\e^{-D}$.  This finishes the proof.
 \qed
 
 \vspace{2mm}
 
 We now provide the proof of Lemma \ref{lem:cover}.
 
  \vspace{2mm}

\noindent{\bf Proof of Lemma \ref{lem:cover}.}We build the partition by induction as follows.

\noindent{\bf \underline{Step 1:}} If there exists some $b_1\in B$ with $|N(b_1)\cap A|\geq \delta^{1/2} |A|$, let $A_1$ be any subset of $N(b_1)\cap A$ of size $\delta^{1/2}|A|$, and go to step 2.  Otherwise, set $t=0$. In this case, we have that $|E|\leq \delta^{1/2}|A||B|$. Consequently, there exists $A_0\subseteq A$ of size at least $(1-\delta^{1/4})|A|$ so that for all $x\in A_0$, $|N(x)\cap B|\leq \delta^{1/4}|B|$.  Now let $A_{err}=A\setminus A_0$ and end the construction. 

\noindent{\bf \underline{Step j+1:}} Suppose by induction we have defined $A_1,\ldots, A_j$.  If $|A\setminus (\bigcup_{j'=1}^jA_{j'})|\geq \delta^{1/4}|A|$ and there exists $b_{j+1}\in B$ with $|N(b_{j+1})\cap (A\setminus (\bigcup_{j'=1}^jA_{j'})|\geq \delta^{1/2}|A\setminus (\bigcup_{j'=1}^jA_{j'}))|$, then let $A_{j+1}$ be any subset of $N(b_{j+1})\cap (A\setminus (\bigcup_{j'=1}^jA_{j'})$ of size $\delta^{1/2}|A\setminus (\bigcup_{j'=1}^jA_{j'}))|$, and go to step $j+2$.  Otherwise, set $t=j$. In this case, $|E\cap K_2[A\setminus (\bigcup_{j'=1}^jA_{j'}),B]|\leq \delta^{1/2}|A\setminus (\bigcup_{j'=1}^jA_{j'})||B|$. Consequently, there exists $A_0\subseteq A\setminus (\bigcup_{j'=1}^jA_{j'})$ of size at least $(1-\delta^{1/4})|A\setminus (\bigcup_{j'=1}^jA_{j'})|$ so that for all $x\in A_0$, $|N(x)\cap B|\leq \delta^{1/4}|B|$.  Now let $A_{err}=A\setminus ((\bigcup_{j'=1}^jA_{j'})\cup A_0)$ and end the construction. 

Clearly this algorithm will halt at some stage $t\leq \delta^{-1/2}$, producing the desired partition.  
\qed

\vspace{2mm}

\bibliography{growth.bib}
\bibliographystyle{amsplain}

\end{document}